\documentclass[11pt,reqno]{amsart}
\usepackage{amssymb,amsfonts,amsthm,amscd,esint}
\usepackage{pdfsync}
\usepackage{amsmath, constants}\multlinegap=0pt
\usepackage[latin1]{inputenc}
\usepackage{amsthm}
\usepackage{amssymb}
\usepackage[francais,english]{babel}

\usepackage{color}

\usepackage{amsfonts,amsmath,layout}

\newcommand{\N}{\mathbb N}

\newcommand{\Z}{\mathbb Z}

\newcommand{\R}{\mathbb R}

\newcommand{\T}{\mathbb{T}}

\def\R{\mathbb R}
\def\N{\mathbb N}
\def\Z{\mathbb Z}

\def\vep{\varepsilon}
\def\rg{\rangle} 
\def\lg{\langle}

\newcommand{\be}{\begin{equation}}
\newcommand{\ee}{\end{equation}}

\def\1{{\bf 1}}
\def\rife#1{(\ref{#1})}
\def\vep{\varepsilon}

\def\de{\delta}
\def\la{\lambda}

\DeclareMathOperator*{\essinf}{ess-inf}

\def\qqf{\qquad\forall}

\def\qmb{\qquad\mbox}
\def\ds{\displaystyle}

\def\vep{\varepsilon}

\newtheorem{Theorem}{Theorem}[section]

\newtheorem{Proposition}[Theorem]{Proposition}
\newtheorem{Lemma}[Theorem]{Lemma}
\newtheorem{Corollary}[Theorem]{Corollary}
\newtheorem{Remark}[Theorem]{Remark}

\begin{document}

\title[Sobolev regularity for the HJ equation]{Sobolev regularity for the first order Hamilton-Jacobi equation}
\author{Pierre Cardaliaguet}
\address{Ceremade, Universit\'e Paris-Dauphine,
Place du Maréchal de Lattre de Tassigny, 75775 Paris cedex 16 - France}
\email{cardaliaguet@ceremade.dauphine.fr }
\author{Alessio Porretta}
\address{Dipartimento di Matematica, Universit\`a di Roma Tor Vergata,
Via della Ricerca Scientifica 1, 00133 Roma (Italy)}
\email{porretta@mat.uniroma2.it}
\author{Daniela Tonon}
\address{Ceremade, Universit\'e Paris-Dauphine,
Place du Mar\'echal de Lattre de Tassigny, 75775 Paris cedex 16 - France}
\email{tonon@ceremade.dauphine.fr}
\thanks{Cardaliaguet was partially supported by the ANR (Agence Nationale de la Recherche) project  ANR-12-BS01-0008-01.}

\dedicatory{Version: \today}

\maketitle

\begin{abstract} We provide Sobolev estimates for solutions of first order Hamilton-Jacobi equations with Hamiltonians which are superlinear in the gradient variable. We also show that the solutions are differentiable almost everywhere. The proof relies on an inverse Hölder inequality. Applications to mean field games are discussed.   
\end{abstract}

\section{Introduction}

The goal of the paper is to establish Sobolev estimates for solutions of first order Hamilton-Jacobi (HJ) equations of the form 
\be\label{HJ0}
\partial_t u +H(t,x,Du)= f(t,x).
\ee
Here we assume that $H$ has a $p-$growth in the gradient variable ($H=H(t,x,\xi)\approx |\xi|^p$ at infinity, with $p>1$) and $f$ is  a continuous map. By Sobolev estimates, we mean estimates of $u$ (in Sobolev spaces) which are independent of the regularity of $H$ and $f$ and depend only on the growth of $H$, the $L^r$ norm of $f$ and the $L^\infty$ norm of $u$.

The main result of this paper is a Sobolev regularity estimate for the solution of \eqref{HJ0} and its almost everywhere (a.e.) differentiability.  As far as we know, these questions have never been addressed before. Besides their intrinsic interest, they are motivated by the theory of mean field games (see Lasry-Lions \cite{lasry06, lasry06a}): our regularity result implies that ``weak solutions" of the mean field game systems satisfy the equation in a more classical sense (see section \ref{sec:MFG}). Indeed, one of the main consequences of our estimate is the gain of compactness for both the time-derivative and the Hamiltonian term in the equation.
\bigskip

We now state our main result. 
For $\rho>0$ let us set $Q_\rho:=(-\rho/2,\rho/2)^d$. Let $f:[0,1]\times \overline{Q_1}\to \R$ be continuous and nonnegative and $u$ be continuous on $[0,1]\times \overline{Q_1}$ and satisfy in the viscosity solutions' sense 
\be\label{ineq:subsolIntro}
\partial_t u +\frac{1}{\bar C} |Du|^p \leq f(t,x)\qquad  {\rm in }\; (0,1)\times Q_1
\ee
and 
\be\label{ineq:supersolIntro}
\partial_t u +\bar C |Du|^p \geq -\bar C\qquad  {\rm in }\; (0,1)\times Q_1
\ee
(for the notion of viscosity solution, see \cite{CIL}). 

\begin{Theorem}\label{th:main} Assume $p>1$ and $r>1+d/p$. Then $u\in W^{1,1}_{loc}((0,1)\times Q_1)$ and, for any $\delta>0$, there exists $\vep>0$ and  $M$ such that 
$$
\|\partial_t u \|_{L^{1+\vep}((\delta,1-\delta)\times Q_{1-\delta})} + \|Du \|_{L^{p(1+\vep)}((\delta,1-\delta)\times Q_{1-\delta})} \leq M, 
$$
where $\vep$ depends on $d$, $p$, $r$ and $\bar C$ while $M$ depends on $d$, $p$, $r$, $\bar C$, $\|f\|_r$, $\|u\|_\infty$ and  $\delta$.

Moreover $u$ is differentiable at almost every point of $(0,1)\times Q_1$.
\end{Theorem}

The result directly  applies to viscosity solutions of \eqref{HJ0} under growth conditions on the Hamiltonian, i.e., if there exists  $\bar C>0$ and $p>1$ such that 
\be\label{igrowth}
 \frac{1}{\bar C} |\xi|^p -\bar C\leq H(t,x,\xi)\leq \bar C |\xi|^p +\bar C.
 \ee
 The estimate is then independent of the regularity of the Hamiltonian.\\
 
  Let us now comment upon the assumptions. The fact that $f$ is nonnegative is of course irrelevant: it is however important that $f$ is bounded below, and in this case the constants $\vep$ and $M$ also depend on this bound.  
  
The result does not hold in general  if $H$ has linear growth in the gradient variable: for instance if $H$ is positively homogeneous in the gradient variable, then the equation is invariant by scaling: the solution is then of bounded variation and the best estimate one can expect is a BV bound. 

In a more subtle way, we do not expect the result to hold if $H$ is coercive but has a different growth from below and from above. Indeed, when $H$ is convex, it is associated with a problem of calculus of variations. If the corresponding Lagrangian has a different growth from above and from below, the optimal trajectories may have a very singular behavior (see \cite{AA} for instance) which seems to be incompatible with Sobolev estimates; for sure the solution is not H\"{o}lder continuous in this case. 

Finally, under the assumptions of the above Theorem, Sobolev regularity only holds for small $\vep$. This point is discussed in section \ref{sec:examples}. Let us underline that a quantification of such a constant is an open problem.  We also note in section \ref{sec:examples} that the result is {\it not} about the regularity of Sobolev functions satisfying some inequalities. Indeed the result does not hold if, for instance, we assume that $u$ is in $W^{1,1}$ and satisfies \eqref{ineq:subsolIntro} and \eqref{ineq:supersolIntro} a.e.: the super-solution inequality has to hold in a viscosity sense. So our result is genuinely about viscosity solutions of Hamilton-Jacobi (in)equations. \\

The idea that solutions of Hamilton-Jacobi equations which are coercive with respect to the  gradient variable enjoy unexpected regularity goes back to Capuzzo~Dolcetta, Leoni and the second author \cite{CDLP} (see also Barles \cite{Ba09}) who proved that subsolutions of stationary HJ equations of second order with super-quadratic growth in the gradient variable have Hölder bounds. The method relies on the construction of a suitable class of supersolutions. The result was later extended to equations with unbounded right-hand side by Dall'Aglio and the second author \cite{DaPo}. For HJ equations of evolution type, Hölder bounds were progressively obtained by the first author (\cite{ca08}, first order case, convex Hamiltonians), then in collaboration with Cannarsa (\cite{CC}, quasilinear case), with Rainer (\cite{cr11}, fully nonlinear, nonlocal equations), and with Silvestre (\cite{CS}, unbounded right-hand side). The proof in this setting is completely different, and more involved, than for the stationary case: it relies either on suitable one-dimensional reverse Hölder inequality (as in \cite{CC, ca08, cr11}) or on the method of improvement of oscillations (as in \cite{CS}). The regularity holds for solutions---but not for subsolutions---of the equation. 

Such results were initially motivated by homogenization, for which estimates of the solution independent of the regularity of the Hamiltonian is necessary. It is however the theory of mean field games which has motivated the analysis for Hamilton-Jacobi equation with possibly unbounded right-hand side: see  \cite{c1}.\\
 
 Some comments on the techniques of the proofs are now in order. Let us first recall the main argument for the {\it Hölder} regularity. In \cite{CC, ca08, cr11} the key point is that ``generalized characteristics" of the equation satisfy a (one-dimensional) reverse Hölder inequality. By generalized characteristics, we mean trajectories along which the solution enjoys a suitable monotonicity property (see Lemma \ref{lem:supersol}). Using a well-known result of Gehring \cite{Ge} (see also Giaquinta-Modica \cite{GiMo}) the reverse Hölder inequality implies an extra regularity for the generalized characteristics, and this in turn can be used to prove the Hölder regularity of the solution. The key feature of this construction is that the regularity of the solution is generated by pointwise estimates of this solution along suitable paths (generalized characteristics). Although quite different in spirit, the proof in  \cite{CS} also relies on the fact that Hamilton-Jacobi equations ``see points". In contrast, Sobolev estimates require {\it integral} bounds, which is not very natural for HJ equations. In particular, while viscosity solutions to  \eqref{ineq:subsolIntro} also satisfy the inequality  in the sense of distributions, a similar statement is not known for viscosity solutions to \eqref{ineq:supersolIntro}. 
 
 The main ingredient of our proof of the Sobolev estimate is still---not surprisingly---a reverse Hölder inequality. However, in contrast with the previous arguments, it is the gradient $Du$ of the solution which satisfies such an inequality---and not the generalized characteristics. To explain why such an inequality is plausible, let us (formally) integrate inequality \eqref{ineq:subsolIntro} satisfied by $u$ over a cube $(t,t+h)\times Q_h(x)$ (for $h>0$ small): we get
$$
\fint_{(t,t+h)\times Q_h(x)} |Du|^p \leq C \fint_{Q_h(x)} \frac{u(t,x)- u(t+h,x)}{h}\ dx + \fint_{(t,t+h)\times Q_h(x)}  f.
$$
In order to get a reverse Hölder inequality, one has to show that the right-hand side is bounded above 
 by an expression of the form $\ds C\left(\fint_{(t,t+h)\times Q_h(x)} |Du|\right)^p+C$. In other words, one has to quantify the oscillation in time of the solution by its oscillation in space. For this we use again the generalized characteristics. The rough idea is that  the solution cannot vary in time unless generalized characteristics bend sufficiently, thus propagating  the information over the space. The main difficulty of course is to find a quantitative way to explain this. We show that the correct encoding of the phenomenon is through an ``energy" of suitable generalized characteristics (see the comment after the reverse Hölder inequality given in Proposition \ref{lem:key}). 

As time and space play at different scales  in inequalities \eqref{ineq:subsolIntro} and \eqref{ineq:supersolIntro}, it is convenient to use ideas introduced by DiBenedetto for degenerate parabolic equations \cite{DiB} and refined by Kinnunen and Lewis \cite{KiLe}. This consists in working on space-time cubes which depend on the solution itself (see Proposition \ref{lem:key} for details). 

Once we know that the solution belongs to a Sobolev space, it is natural to ask for differentiability. We show that the solution is a.e. differentiable by a blow up argument at points of approximate differentiability: the (in)equations \eqref{ineq:subsolIntro} and \eqref{ineq:supersolIntro} being (almost) invariant by blow up, the H\"older estimates are valid at each scale, which provides the result.  \\

 Let us finally explain the organization of the paper: in section \ref{sec:prem} we recall basic results for subsolutions of \eqref{ineq:subsolIntro} and  supersolutions of \eqref{ineq:supersolIntro}. In particular we discuss the notion of generalized characteristics. In section \ref{sec:keykey}  we explain that $|Du|$ satisfies a  reverse H\"{o}lder inequality (Proposition \ref{lem:key}), which allows us to  prove the Sobolev estimate in the next section. Section \ref{sec:aediff} is devoted to the a.e. differentiability of the solution. We complete the paper by showing through an example that one cannot expect the result of the theorem to hold for large $\vep$ (section \ref{sec:examples}) and then by an application to mean field games (section \ref{sec:MFG}). \\

{\bf Notation:} Throughout the paper, we let, for $\rho>0$, $Q_\rho:=(-\rho/2,\rho/2)^d$ and, for $\sigma,\rho>0$, $\ds Q_{\sigma,\rho}= (-\sigma/2, \sigma/2)\times (-\rho/2, \rho/2)^d$. If $(t,x)\in \R^{d+1}$, we set $Q_{\sigma,\rho}(t,x)= (t-\sigma/2, t+\sigma/2)\times \prod_{i=1}^d (x_i-\rho/2, x_i+\rho/2)$. For $x\in \R^d$ and $r>0$, $B(x,r)$ denotes the closed ball centered at $x$  of radius $r$; we shorten   $B(0,1)$ into $B_1$ . If $E$ is a measurable subset of $\R^n$ ($n\geq 1$),  $|E|$ stands for its measure while, if $f:E\to \R$ is integrable, we denote by $\ds \fint_E f:= \frac{1}{|E|} \int_E f$ the average of $f$ on $E$. 

\section{Preliminaries}\label{sec:prem}

In this section we recall several known facts about inequalities  \eqref{ineq:subsolIntro} and \eqref{ineq:supersolIntro}. These results are related to the Lax formula and the comparison principle. For simplicity, we work from now on with {\it backward } Hamilton-Jacobi equations, i.e., with a continuous map $u$ which satisfies the following inequalities in the viscosity sense: 
\be\label{ineq:subsol}
-\partial_t u +\frac{1}{\bar C} |Du|^p \leq f(t,x)\qquad  {\rm in }\; (0,1)\times Q_1
\ee
and 
\be\label{ineq:supersol}
-\partial_t u +\bar C |Du|^p \geq -\bar C\qquad  {\rm in }\; (0,1)\times Q_1\,.
\ee
We recall that, throughout the paper, $p>1$ and  $r>1+d/p$ are given. We denote by $q$ the conjugate exponent of $p$: $1/p+1/q=1$. The purpose of  next two statements is to  point out how the value of $u$ at a point $(t,x)$ can be compared with the value at different points at times $s>t$.

Let us start first with   a consequence of inequality \eqref{ineq:subsol}.

\begin{Lemma}\label{lem:subsol} Fix $r_1\in (1+d/p,r)$, $\bar \alpha>0$ and $h>0$ such that $2\bar \alpha h<1$. If $u$ satisfies \eqref{ineq:subsol} in the viscosity sense, then for any $(t,x), (s,y) \in (0,1)\times Q_{\bar \alpha h}$ with $s>t$, 
$$
u(t,x)\leq u(s,y) +C \frac{(\bar \alpha h)^q}{(s-t)^{q-1}} + C(s-t) \left(\fint_{(s,t)\times Q_{2\bar \alpha h}} f^{r_1}\right)^{1/r_1},
$$
where $1/p+1/q=1$ and $C= C(p,\bar C)$.
\end{Lemma}

\begin{proof} We follow mostly \cite{cg}. For $\sigma\in B_1$,  $\beta \in (\frac 1p, \frac{r_1-1}d)$ and $\delta:= \bar \alpha h (s-t)^{-\beta}$ let us set 
$$
\gamma_\sigma(\tau)= \left\{\begin{array}{ll}
 \gamma_0(\tau)+ \delta\sigma (\tau-t)^\beta & {\rm if } \; \tau \in [t, (t+s)/2]\\
 \gamma_0(\tau)+ \delta\sigma (s-\tau)^\beta& {\rm if } \; \tau \in [(t+s)/2, s]
\end{array}\right.
$$
where $\gamma_0(\tau)= x+ (\tau-t)(y-x)/(s-t)$. We note that $\gamma_\sigma(\tau)\in Q_{2\bar \alpha h}$ thanks to the choice of $\delta$ and the fact that $x,y\in Q_{\bar \alpha h}$. 

From standard comparison principle and representation formula for Hamilton-Jacobi (see \cite{CC}) we have 
$$
u(t,x)\leq u(s,y) +\int_t^s (C |\dot \gamma_\sigma(\tau)|^q +f(\tau, \gamma_\sigma( \tau)))\ d\tau.
$$
So 
$$
u(t,x)\leq u(s,y) +|B_1|^{-1}\int_{B_1}\int_t^s (C |\dot \gamma_\sigma(\tau)|^q +f(\tau, \gamma_\sigma( \tau)))\ d\tau d\sigma.
$$
Following \cite{cg}, we have 
$$
\begin{array}{rl}
\ds 
\int_t^s |\dot \gamma_\sigma(\tau)|^q\; \leq & \ds C\frac{|x-y|^q}{(s-t)^{q-1}}+ C\delta^q\left( \int_t^{(t+s)/2} (\tau-t)^{q(\beta-1)}d\tau+\int_{(t+s)/2}^s (s-\tau)^{q(\beta-1)}d\tau\right)\\
\leq &  \ds C\frac{|x-y|^q}{(s-t)^{q-1}}+ C\delta^q(s-t)^{q(\beta-1)+1} \leq  C\frac{|x-y|^q}{(s-t)^{q-1}}\; \leq \; C\frac{(\bar \alpha h)^q}{(s-t)^{q-1}}, 
\end{array}
$$
where we used  $\beta>\frac1p$ to ensure the convergence of the integrals and the definition of $\delta$ for the last inequality. 

On the other hand, using a change of variable and then Hölder's inequality, we have
$$
\begin{array}{l}
\ds  |B_1|^{-1}\int_{B_1}\int_t^{(t+s)/2} f(\tau, \gamma_\sigma( \tau))\ d\tau d\sigma \\ 
\qquad  \ds 
= C\int_t^{(t+s)/2} \int_{B(\gamma_0(\tau), \delta (\tau-t)^\beta)} \delta^{-d}(\tau-t)^{-\beta d} f(\tau, z)\ dzd\tau \\
\qquad \leq  \ds  C \left( \int_t^{(t+s)/2} \int_{B(\gamma_0(\tau), \delta (\tau-t)^\beta)} \delta^{-dr_1'}(\tau-t)^{-\beta dr_1'}\right)^{1/r_1'}
\left(\int_t^s \int_{Q_{2\bar \alpha h}} f^{r_1}\right)^{1/r_1}
\end{array}
$$
(we used the definition of  $\delta$ to ensure that $B(\gamma_0(\tau), \delta (\tau-t)^\beta)\subset Q_{2\bar \alpha h}$). As,  from the choice of  $\delta$ and using $\beta<\frac{r_1-1}d$,
$$
\begin{array}{rl}
\ds \left( \int_t^{(t+s)/2} \int_{B(\gamma_0(\tau), \delta (\tau-t)^\beta)} \delta^{-dr_1'}(\tau-t)^{-\beta dr_1'}\right)^{1/r_1'} 
= & \ds  C\delta^{-d/r_1}(s-t)^{-\beta d/r_1+ 1/r_1'}\\ 
= & \ds C (\bar \alpha h)^{-d/r_1} (s-t)^{1-1/r_1},
\end{array}
$$
we get
$$
\ds  |B_1|^{-1}\int_{B_1}\int_t^{(t+s)/2} f(\tau, \gamma_\sigma( \tau))\ d\tau d\sigma  \leq C(s-t) \left(\fint_{(t,s)\times Q_{2\bar \alpha h}} f^{r_1}\right)^{1/r_1}
$$
We can argue in the same way on the time-interval $((t+s)/2, s)$ to get the result. 
\end{proof}

The next Lemma, which is concerned with inequality  \eqref{ineq:supersol}, is standard and can be found, for instance, in \cite{CC}, Lemma 4.3.

\begin{Lemma}\label{lem:supersol} If $u$ is continuous on $[0,1]\times \overline{Q_1}$ and satisfies \eqref{ineq:supersol} in the viscosity sense, then, for any $(t,x)\in (0,1)\times Q_1$, there exists an absolutely continuous curve $\gamma$ with $\gamma(t)=x$ and, for any $s\in [t,1]$ such that $\gamma([t,s])\subset Q_1$,  
$$
u(t,x)\geq u(s,\gamma(s)) +\frac{1}{C} \int_t^s |\dot \gamma(\sigma)|^qd\sigma - C(s-t) ,
$$
where $C= C(p,\bar C)$. 
\end{Lemma}

We say that $\gamma$ is a generalized characteristic for $u(t,x)$. Indeed, if $u$ is a solution of a Hamilton-Jacobi-Belmann equation, then any characteristic $\gamma$ satisfies the above inequality. 
\vskip0.5em

Next  we explain that inequality \eqref{ineq:subsol} can also be understood in the distributional sense: 

\begin{Lemma}\label{lem:distrib} If $u$ is continuous on $[0,1]\times \overline{Q_1}$ and satisfies \eqref{ineq:subsol} in the viscosity sense, then $u$ is of bounded variation (BV) in $(0,1)\times Q_1$, $Du\in L^p((0,1)\times Q_1)$ and \eqref{ineq:subsol} holds in the sense of distributions. 
\end{Lemma}

\begin{proof} For $\vep>0$ small, let $u^\vep$ be the standard sup-convolution of $u$ (see \cite{CIL}). By continuity of $f$, $u^\vep$ is still a subsolution of the approximate inequality 
\be\label{ineq:subsolapprox}
-\partial_t u^\vep +\frac{1}{\bar C} |Du^\vep|^p \leq f(t,x)+\delta\qquad  {\rm in }\; (\delta,1-\delta)\times Q_{1-\delta}
\ee
where $\delta\to 0^+$ as $\vep\to 0$. As $u^\vep$ is a Lipschitz continuous subsolution, the above inequality is satisfied a.e. and therefore in the sense of distributions. Integrating \eqref{ineq:subsolapprox} in time-space shows that $Du^\vep$ is bounded in $L^p$. Hence $Du$ also belongs to $L^p$. Note that $\partial_t u^\vep$ is bounded below by $-\|f\|_\infty-1$ for $\vep$ small enough. As $u^\vep$ is uniformly bounded, $\partial_t u^\vep$ is bounded in $L^1$. Therefore $u^\vep$ is bounded in BV. Hence  $u$ belongs to BV. Using the fact that $u^\vep$ satisfies inequality \eqref{ineq:subsolapprox} in the sense of distributions, and that the nonlinearity is convex, we finally obtain that  \eqref{ineq:subsol} holds in the sense of distributions.
\end{proof}

\section{Key estimate}\label{sec:keykey}

Throughout this part $u: [-1/2,1/2]\times Q_1\to \R$ is a continuous map which satisfies  \eqref{ineq:subsol} and  \eqref{ineq:supersol} in  $[-1/2,1/2]\times Q_1$. Our aim is to show that, if $Du$ and $f$ are well estimated in some cube, then $Du$ satisfies a reverse Hölder inequality. To this purpose, we will need to use cubes with an intrinsic scaling.

Let us then introduce a family of parameters. As before, $p>1$ and  $r>1+d/p$ are fixed and we denote by $q$ the conjugate exponent of $p$: $1/p+1/q=1$. We also fix $r_1\in (1+d/p, r)$.  For constants $\lambda_0\geq 1$,  $\kappa\geq 1$ and $2\leq c_1\leq 5 c_1\leq c_2$ to be defined below and variables $\lambda\geq \lambda_0$ and $h>0$, we set 
$$
\sigma = \kappa\lambda^{1-p} 
$$
and  
$$
Q=Q_{\sigma h,  h}, \; Q'= Q_{c_1\sigma  h, c_1  h}, \; Q''= Q_{c_2\sigma h, c_2  h},
$$
where the cubes are centered at some generic point $(t_0,x_0)\in Q_{1,1}$, which for simplicity we may assume to be the origin $(0,0)$.
We also assume that  $ Q''\subset Q_{1, 1}$. 

The main result of this part is the following statement. 

\begin{Proposition}\label{lem:key} There exists a suitable choice of the constants $\lambda_0$, $\kappa$, $c_1$, $c_2$, depending only on $d$, $p$, $r_1$, $r$ and $\bar C$ such that, for any $\lambda\geq \lambda_0$ and $h>0$, if 
the following estimate holds:
\be\label{keyhyp}
\lambda^p \leq  \fint_{Q}(|Du|^p+ f^{r_1}) \leq c_2^{d+1} \fint_{Q''}(|Du|^p+ f^{r_1})\leq c_2^{d+1} \lambda^p, 
\ee
then we have
\be\label{keycl}
\fint_{Q''}|Du|^p \leq \hat C \left(\fint_{Q'} |Du|\right)^p+ \hat C \fint_{Q'} (1 + f^{r_1}),
\ee
 for some constant $\hat C$ independent of $\la, h$.
\end{Proposition}

The idea to consider cylinders with a size depending on the solution itself goes back to DiBenedetto \cite{DiB}. We borrow the precise formulation of Proposition \ref{lem:key} to the seminal paper by Kinnunen and Lewis \cite{KiLe} on parabolic systems of p-Laplacian type. The proof of the Lemma, however, is completely different from \cite{KiLe} since Hamilton-Jacobi equations and p-Laplace systems do not behave  at all in the same way. \\

The rest of the subsection is devoted to the proof of Proposition \ref{lem:key}. Since this will require several steps and technical details, let us try first to explain the general strategy. The starting point consists in integrating (in)equation \eqref{ineq:subsol} over the cube $Q$ to get: 
$$
\fint_{Q} |Du|^p \leq C \fint_{Q_h} \frac{u(\sigma h/2)- u(-\sigma h/2)}{\sigma h} + C \fint_{Q} f.
$$
The difficult part amounts to estimate the first term in the right-hand side by $\ds \left(\fint_{Q'} |Du|\right)^p$, that is to estimate the variation in time of $u$ by its variation in space. The main steps towards this goal are Lemmata \ref{lem:pointcle} and \ref{lem:intutauu0bis}, which show that the two quantities discussed above are both estimated by the energy $\xi:=\ds \int_t^{t+\tau} |\dot \gamma|^q$  of a generalized characteristic starting from a suitable position $(t,x)$. The choice of $(t,x)$ will come from  Lemma \ref{lem:choixdet} below.  

Let us stress once more the  crucial role played by this curve in our strategy. As you see from Section 2, if we need to estimate from above $u(\sigma h/2, \cdot)$, we can do that generically only with points at  a larger time (Lemma \ref{lem:subsol}). On the other hand, we can estimate from above $u$  with {\em some} value at previous times if we move back along a generalized characteristic (Lemma \ref{lem:supersol}). Therefore,  our strategy will be to catch a generalized characteristic $\gamma$ going from some point $(t,x)$ with $t<-\sigma h/2$ to some point $(t+\tau, \gamma(t+\tau))$ with $t+\tau>\sigma h/2$; in this way, we will, roughly speaking, estimate $u(\sigma h/2, \cdot)$ with $u(t+\tau, \gamma(t+\tau) )$ through Lemma \ref{lem:subsol}, then estimate $u(t+\tau, \gamma(t+\tau) )$ with $u(t,x)$ along the characteristic by Lemma \ref{lem:supersol}, and finally use  again  Lemma \ref{lem:subsol} to estimate $u(t,x)$ with  $u(-\sigma h/2, \cdot)$ since $t<-\sigma h/2$. Of course, we also need to check that catching  this characteristic curve is not taking us too far in space; indeed, the space variations of those points will be also measured through the bending of the curve. This explains why the energy $\xi:=\ds \int_t^{t+\tau} |\dot \gamma|^q$ will be a  good reference to estimate both time and space oscillations.

\vskip1em
Let us now proceed into the technical steps. We assume without loss of generality that $f\geq 1$.  Beside $\lambda_0$, $\kappa$, $c_1$, $c_2$, we introduce two other constants $\vep, \delta>0$ small.  Let us explain how we choose the various constants in order to avoid circular arguments. We will successively define $c_1$, $c_2$, $\kappa$, $\vep$, $\delta$  and $\lambda_0$. In other words,  $\lambda_0$ is a  (large)  function of $c_1$, $c_2$, $\kappa$, $\vep$ and $\delta$, $\delta$ is a (small) function of  $c_1$, $c_2$, $\kappa$ and $\vep$, etc... Throughout this part, $C$ denotes a generic constant, which varies from line to line and depends on $d$, $p$, $r_1$, $r$ and $\bar C$ but not on $h$, $\lambda$, $c_1$, $c_2$, $\kappa$, $\delta$, $\vep$ or $\lambda_0$.  

We first note that, if $\ds \delta \fint_{Q}|Du|^p <  \fint_{Q'} f^{r_1}$, then \eqref{keycl} holds in an obvious way. Indeed, by our assumption \eqref{keyhyp}, 
$$
\fint_{Q''}|Du|^p\leq \lambda^p\leq \fint_{Q}(|Du|^p+ f^{r_1})\leq
(\delta^{-1}+c_1^{d+1})  \fint_{Q'} f^{r_1}. 
$$
So we can assume from now on that 
\be\label{casetotreat}
\ds \fint_{Q'} f^{r_1}\leq  \delta \fint_{Q}|Du|^p.
\ee

The first step consists in building a suitable time at which we can fully exploit conditions \eqref{keyhyp}. 

\begin{Lemma}\label{lem:choixdet}
There exists $t\in (-\sigma h, -\sigma h/2)$, which is a Lebesgue point of $\ds s\to \fint_{Q_{c_1 h}} |Du(s)|^p$, such that 
\be\label{choixdet}
\fint_{Q_{c_1 h}} |Du(t)|^p \leq Cc_2^{d+1}  \lambda^p\qquad {\rm and} \qquad \sup_{\tau\in (0,c_2\sigma h/2)} \fint_{t-\tau}^{t+\tau} \fint_{ Q_{c_2 h}} f^{r_1} \leq C c_2\lambda^p. 
\ee
\end{Lemma}

\begin{proof}
To prove the claim, let us set $\ds g(t)=  \fint_{Q_{c_2 h}} f^{r_1}(t)$ if $t\in (-c_2\sigma h/2,c_2\sigma h/2)$ and $g(t)=0$ otherwise. We introduce the maximal function $Mg$ associated with $g$ defined as $\ds Mg(t)= \sup_{\tau>0} \fint_{t-\tau}^{t+\tau} g(s)ds$. Then it is known (cf. \cite{stein}) that, for any $\alpha>0$,  
$$
\left| \left\{ t\in \R\;, \; Mg(t)\geq \alpha\right\}\right| \leq \frac{5}{\alpha} \|g\|_1. 
$$
As, by \eqref{keyhyp}, 
$$\ds \|g\|_1= c_2 \sigma  h \fint_{Q''} f^{r_1} \leq c_2\sigma  h  \lambda^p,
$$ we obtain
$$
\left| \left\{ t\in (-\sigma h, -\sigma h/2)\;, \; Mg(t)\geq C c_2 \lambda^p \right\}\right| \leq \sigma h/8. 
$$
On the other hand, as, by \eqref{keyhyp},  
$$
\ds 
\int_{-\sigma h}^{-\sigma h/2} \fint_{Q_{c_1 h}} |Du|^p\leq Cc_2^{d+1} \sigma h \fint_{Q''} |Du|^p\leq Cc_2^{d+1} \sigma h  \lambda^p,
$$ 
we have 
$$
\left| \left\{ t\in (-\sigma h, -\sigma h/2)\;, \; \fint_{Q_{c_1h}} |Du(t)|^p \geq Cc_2^{d+1} \lambda^p \right\}\right| \leq \sigma h/8. 
$$
Combining the two inequalities above shows the existence of a time $t\in  (-\sigma h, -\sigma h/2)$ such that \eqref{choixdet} holds. 
\end{proof}

Fix from now on $t$ as in Lemma \ref{lem:choixdet} and choose $x\in Q_h$ such that 
$\ds u(t,x)=\fint_{Q_h} u(t, y)\ dy$ and consider a generalized characteristic   $\gamma$  for $u(t,x)$, i.e., such that 
\be\label{quasi-minbis}
u(t,x)\geq u(s,\gamma(s)) +\frac{1}{C} \int_t^s |\dot \gamma(r)|^q\ dr  - C(s-t) \qquad \forall s\geq t.  
\ee
Recall that such a generalized characteristic exists thanks to Lemma \ref{lem:supersol}. Let us now define 
$$
\tau := \sup \left\{ t'\in (0,c_2\sigma h]: \gamma(s)\in Q_{c_1 h/2} \; \forall s\in [t, t+t']\right\}.
$$
So $\gamma(s)$ remains in $Q_{c_1 h/2}$ for $s\in [t,t+\tau]$ and, by definition, either $|\gamma(t+\tau)|= c_1 h/2$ or $\tau = c_2 \sigma h$. 
We set 
$$
\ds \xi:= \int_{t}^{t+\tau}|\dot \gamma(r)|^q \ dr\,.
$$ 
The quantity $\xi$ plays a major role in the next analysis since it allows  us to quantify the distortion of the map $u$ as explained through Lemmata \ref{lem:pointcle} and \ref{lem:intutauu0bis}. 

Let us first give some estimates on the time $\tau$ and the quantity $\xi$.

\begin{Lemma}\label{lem:cas2bis} We have
\be\label{cas2bis}
\tau<c_2 \sigma h/4\,.
\ee
Moreover we can choose $\la_0$ sufficiently large (depending on $c_1$, $c_2$, $\kappa$, $\vep$ and $\delta$) in order to have
\be\label{cas2bisbis}
\tau \left(\fint_{(t,t+\tau)\times Q_{c_1 h}} f^{r_1}\right)^{1/r_1} \leq \vep \xi\;.
\ee
\end{Lemma}

\vskip0.5em
Let us recall that $t$, as well as $\tau$, depend on all parameters $c_1, c_2, \kappa, \de,\lambda$ with $\la\geq \la_0$, and that $\la_0$ is the last parameter to be chosen once  the others have been fixed. In particular, as we assume $f\geq 1$, \rife{cas2bisbis} implies: 
\be\label{barctaubis}
\tau\leq \vep \xi.
\ee

\begin{proof}
To prove \rife{cas2bis}, we argue by contradiction, assuming $\tau\geq c_2 \sigma h/4$. Then
$$
t+\tau-\sigma h/2 \geq (c_2/4- 3/2)\sigma h\geq   c_2\sigma h/8
$$
for $c_2$ large enough, while 
$$
t+\tau-\sigma h/2\leq  \tau\leq c_2 \sigma h \,
$$
by definition of $\tau$.
With this in mind, we use Lemma \ref{lem:subsol} for $y\in Q_h$ and  $\bar \alpha= c_1$: 
$$
\begin{array}{l}
\ds u(\sigma h/2,y) \\
\quad  \leq  \ds u(t+\tau , \gamma(t+\tau)) + C \frac{(c_1h)^q}{(t+\tau-\sigma h/2)^{q-1}}
+ C (t+\tau-\sigma h/2)\left(\fint_{(\sigma h/2, t+\tau)\times Q_{2c_1 h}} f^{r_1}\right)^{1/r_1} \\
\quad \leq  \ds u(t,x) +C\tau + C \frac{(c_1 h)^q}{(c_2 \sigma h)^{q-1}}
+ C c_2^{1+d/r_1} \sigma h \left(\fint_{Q''} f^{r_1}\right)^{1/r_1}
\end{array}
$$
where the last inequality comes from \eqref{quasi-minbis}. Recalling \eqref{keyhyp} and the choice of $x$, and since $\tau\leq c_2\sigma h$,
$$
\begin{array}{rl}
\ds u(\sigma h/2,y) \; \leq & \ds \fint_{Q_h}u(t) +  C c_2 \sigma h + C c_1^qc_2^{1-q} \sigma^{1-q} h+ C c_2^{1+d/r_1} \sigma h \lambda^{p/r_1}\\
\leq & \ds \fint_{Q_h}u(t) + C c_2 \kappa \la h \la_0^{-p}+ C c_1^qc_2^{1-q}\kappa^{1-q} \lambda h+ C c_2^{1+d/r_1}  \kappa\lambda h\lambda_0^{-p(1-1/r_1)}
\end{array}
$$
where we have used the fact that $ \sigma=\kappa \lambda^{1-p}$ and $\lambda\geq \lambda_0$. So for $\lambda_0$ large enough (depending on $c_1$, $\kappa$ and $c_2$), we get
$$
\ds u(\sigma h/2,y) \; \leq \; \ds \fint_{Q_h}u(t) + C c_1^qc_2^{1-q}\kappa^{1-q}  \lambda h.
$$
Since inequality \eqref{ineq:subsol} is satisfied in the sense of distributions (Lemma \ref{lem:distrib}), we can integrate it over $(t, -\sigma h/2)\times Q_{h}$ to estimate $\ds \fint_{Q_h}u(t)$ by $\ds \fint_{Q_h}u(-\sigma h/2)$:
$$
\begin{array}{rl}
\ds \fint_{Q_h}u(t) \; \leq & \ds \fint_{Q_h}u(-\sigma h/2) + C\int_t^{-\sigma h/2}\fint_{Q_{h}}f \\
\leq & \ds  \fint_{Q_h}u(-\sigma h/2) + Cc_2^{(d+1)/r_1} \sigma h \left(\fint_{Q''}f^{r_1}\right)^{1/r_1}
\; \leq \; 
 \fint_{Q_h}u(-\sigma h/2) +  Cc_2^{(d+1)/r_1} \sigma h \lambda^{p/r_1}, 
\end{array}
$$
where we have used \eqref{keyhyp} in the last inequality. 
For $\lambda_0$ large, we can put the above inequalities together and derive 
$$
\fint_{Q_h} u(\sigma h/2)- u(-\sigma h/2) \leq C c_1^qc_2^{1-q}\kappa^{1-q} \lambda h.
$$
Integrating again inequality \eqref{ineq:subsol}, this time on $Q$, gives 
$$
\fint_{Q} |Du|^p \leq C \fint_{Q_h} \frac{u(\sigma h/2)- u(-\sigma h/2)}{\sigma h} + C \fint_{Q} f \leq C c_1^qc_2^{1-q}\kappa^{-q} \lambda^p,
$$
where we have absorbed the term $\ds \fint_{Q} f$ as above. Since, by \eqref{casetotreat}, 
$$
\ds \fint_{Q} f^{r_1}\leq c_1^{d+1}\fint_{Q'} f^{r_1}\leq  c_1^{d+1} \delta \fint_{Q}|Du|^p,
$$
the first inequality in \eqref{keyhyp} implies
$$
\frac1{1+c_1^{d+1}\delta} \lambda^p \leq \fint_{Q} |Du|^p \leq C c_1^qc_2^{1-q} \kappa^{-q} \lambda^p,
$$
which is impossible for $\delta$ small and $c_2$ large enough (depending on $c_1$ and $\kappa$). This shows  \eqref{cas2bis}. The proof of \rife{cas2bisbis} is easier: we now know by the choice of $\tau$ and since $\tau< c_2\sigma h$ that $|\gamma (t+\tau)|=c_1 h/2 $. So, as $x\in Q_h$ and $c_1\geq 2$, 
\be\label{xitaulambda}
c_1 h /4 \leq |\gamma (t+\tau)-x| \leq  \int_{t}^{t+\tau} |\dot \gamma| \leq \xi^{1/q}\tau^{1/p}. 
\ee
Thus, on the one hand, 
$$
\xi\geq C^{-1} \tau (c_1 h\tau^{-1})^q> C^{-1} \tau  c_1^q c_2^{-q} \kappa^{-q} \lambda^p,
$$
where we have used $\tau< c_2\sigma h$ and $\sigma = \kappa\lambda^{1-p}$ to get the last inequality.  In particular, we deduce that
\be\label{khbezoerl}
\tau < K(c_1, c_2, \kappa) \,  \la^{-p}\, \xi
\ee
for a constant $K$ depending on $c_1,c_2,\kappa$. 
On the other hand, by \eqref{choixdet},  
$$
 \left(\fint_{(t,t+\tau)\times Q_{c_1 h}} f^{r_1}\right)^{1/r_1} \leq  Cc_2^{d/r_1}\left( \fint_{t-\tau}^{t+\tau}\fint_{Q_{c_2 h}} f^{r_1}\right)^{1/r_1}\leq
 Cc_2^{d/r_1} (Cc_2\lambda^p)^{1/r_1}
$$
so using  \eqref{khbezoerl} we infer that
$$
\tau \left(\fint_{(t,t+\tau)\times Q_{c_1 h}} f^{r_1}\right)^{1/r_1} \leq  \vep \xi
$$
provided  $\lambda_0$ is large enough. This shows \eqref{cas2bisbis}. 
\end{proof}

Next we estimate the relation between $\tau$, $\xi$ and $h$ slightly more carefully: 
\begin{Lemma}\label{lemineq:htauxibis} We have
\be\label{ineq:htauxibis}
\frac{1}{C} \frac{(c_1 h)^q}{\tau^{q-1}} \leq \xi \leq C \frac{(c_1 h)^q}{\tau^{q-1}}.
\ee
\end{Lemma}

\begin{proof} The first  inequality in \eqref{ineq:htauxibis} can be directly deduced from \eqref{xitaulambda}. For the second one, we have by \eqref{quasi-minbis}
$$
u(t+\tau, \gamma(t+\tau)) \leq u(t,x) -\frac{1}{C} \xi+ C\tau
$$
and on account of \eqref{barctaubis} we get
\be\label{back}
u(t+\tau, \gamma(t+\tau))  \leq u(t,x) -\frac{1}{2C} \xi,
\ee
up to choosing $\la_0$ large enough. 
Then Lemma \ref{lem:subsol} implies that
$$
\begin{array}{rl}
\ds u(t,x)\; \leq & \ds  u(t+\tau, \gamma(t+\tau)) + C \frac{(c_1 h)^q}{\tau^{q-1}} + C\tau\left(\fint_{(t,t+\tau)\times Q_{c_1 h}} f^{r_1}\right)^{1/r_1} \\
\leq & \ds u(t,x)-\frac{1}{2C} \xi + C \frac{c_1^qh^q}{\tau^{q-1}} + C\tau\left(\fint_{(t,t+\tau)\times Q_{c_1 h}} f^{r_1}\right)^{1/r_1} \leq \;  u(t,x)-\frac{1}{3C} \xi + C \frac{c_1^qh^q}{\tau^{q-1}}
\end{array}
$$
thanks to \eqref{cas2bisbis}. This implies the right-hand inequality in \eqref{ineq:htauxibis}. 
%
\end{proof}

In the next step, we explain that $u(s,\cdot)$ is small  in a neighborhood of $\gamma(t+\tau)$ for any $s\in [t,t+\tau/2]$: 

\begin{Lemma}\label{lem:utyBbis} There exists $C_0\geq 1$, depending only on $d$, $p$, $r_1$, $r$ and $\bar C$, such that
\be\label{utyBbis}
 u(s,y)\leq u(t,x)-C_0^{-1} \xi \qquad \forall  (s,y)\in [t,t+\tau/2] \times B(\gamma(t+\tau), h).
 \ee
\end{Lemma} 

We need to keep track of the constant  in \eqref{utyBbis} for the proof of Lemma \ref{lem:pointcle} below. This is the reason why we single it out by the notation $C_0$. 
 
\begin{proof} If $(s,y)\in [t,t+\tau/2]\times B(\gamma(t+\tau), h)$, we have by Lemma \ref{lem:subsol} applied with $\bar \alpha=1$, 
$$
\begin{array}{rl}
u(s,y) \leq & \ds u(t+\tau, \gamma(t+\tau)) + C \frac{h^q}{\tau^{q-1}} + C\tau \left(\fint_{(s,t+\tau)\times Q_{2 h}} f^{r_1}\right)^{1/r_1} \\
\leq & \ds u(t+\tau, \gamma(t+\tau)) + C \frac{h^q}{\tau^{q-1}} + Cc_1^{d/r_1} \tau \left(\fint_{(t,t+\tau)\times Q_{ c_1 h}} f^{r_1}\right)^{1/r_1} \\
\leq & \ds u(t,x)-\frac{1}{2C} \xi + C c_1^{-q}\xi  + Cc_1^{d/r_1} \tau \left(\fint_{(t,t+\tau)\times Q_{ c_1 h}} f^{r_1}\right)^{1/r_1} 
\end{array}
$$
where we used  \rife{back} and inequality \eqref{ineq:htauxibis}. The last two terms now can be absorbed by choosing $c_1$ large enough and Lemma \ref{lem:cas2bis} respectively. So we obtain \rife{utyBbis} for some $C_0$ depending only on $d$, $p$, $r_1$, $r$ and $\bar C$.
\end{proof}

The next step is central in the proof. It consists in showing that the quantity $\xi$ controls below  the oscillation in space of $u$: 
\begin{Lemma}\label{lem:pointcle} For any $s \in [t, t+\tau/2]$ such that $Du(s)\in L^p(Q_{c_1 h})$, we have 
\be\label{pointcle}
 \xi \;  \leq \; Cc_1^{p(d-1)} \tau \left( \fint_{Q_{c_1 h}} |Du(s)|\right)^p.
\ee
\end{Lemma}

\begin{proof}
Let $s\in [t, t+\tau/2]$ and set $\ds \mu_s:= \fint_{Q_{c_1 h}} u(s,y)dy$. Then, by \eqref{utyBbis}, 
$$
\begin{array}{l}
\ds \int_{Q_{c_1 h}} |u(s,y)- \mu_s|^{d/(d-1)}dy \\
\qquad  \geq  \ds \max\left\{ \int_{B(\gamma(t+\tau), h)} (\mu_s - u(s,y))_+^{d/(d-1)} dy\ ; \ 
\int_{Q_h} (u(s,y) -\mu_s)_+^{d/(d-1)}dy \right\}\\
\qquad \geq   \ds \max\left\{ \int_{B(\gamma(t+\tau), h)} (\mu_s - u(t,x)+ C_0^{-1} \xi)_+^{d/(d-1)}\ ; \ h^d
(\fint_{Q_h}  u(s) -\mu_s )_+^{d/(d-1)}\right\}
\end{array}
$$
We now consider two cases, according to whether $\mu_s$ is small or not.  

If $\ds \fint_{Q_h}  u(s) -\mu_s\geq (3C_0)^{-1} \xi$, then clearly
 $$
h^d
(\fint_{Q_h}  u(s) -\mu_s )_+^{d/(d-1)} \geq (3C_0)^{-d/(d-1)} \xi^{d/(d-1)}h^d.
$$
We now suppose that $\ds \mu_s > \fint_{Q_h}  u(s)-(3C_0)^{-1} \xi$. By \eqref{ineq:subsol}  we have
$$
\begin{array}{rl}
\ds \fint_{Q_h}  u(s,y)dy \; \geq & \ds  \fint_{Q_h}  u(t,y)dy - \int_t^s \fint_{Q_h} f \geq u(t,x) - Cc_1^{d/r_1} \tau \left(\fint_{t}^{t+\tau}\fint_{Q_{c_1 h}} f^{r_1}\right)^{1/r_1} 
\end{array}
$$
where we used that $\ds u(t,x)= \fint_{Q_h}  u(t)$. Hence 
$$
\mu_s > \fint_{Q_h}  u(s)-(3C_0)^{-1} \xi \geq u(t,x) - (3C_0)^{-1}\xi- Cc_1^{d/r_1} \tau \left(\fint_{t}^{t+\tau}\fint_{Q_{c_1 h}} f^{r_1}\right)^{1/r_1}.
$$
By Lemma \ref{lem:cas2bis}, last term can be made arbitrarily small compared to $\xi$; so, in this second case we also get
$$
\begin{array}{l}
 \ds \int_{B(\gamma(t+\tau), h)} (\mu_s - u(t,x)+ C_0^{-1} \xi)_+^{d/(d-1)}
\geq   \ds (CC_0)^{-d/(d-1)} \xi^{d/(d-1)} h^d.
\end{array}
$$
Combining the two cases, we deduce 
$$
\ds \int_{Q_{c_1 h}} |u(s,y)- \mu_s|^{d/(d-1)}dy \; \geq \; C^{-1}(C_0)^{-d/(d-1)} \xi^{d/(d-1)} h^d,
$$
which can be rewritten as
 $$
 \left\| u(s,\cdot)- \mu_s\right\|_{L^{d/(d-1)}(Q_{c_1 h})}  \geq C^{-1}(C_0)^{-1} \xi h^{d-1}\; \geq \; C^{-1} h^d c_1 \xi^{1/p}\tau^{-1/p}, 
$$
where we have used \eqref{ineq:htauxibis} for the last inequality and we no longer need to keep track of the notation $C_0$.  We now bound above the left-hand side of the above inequality. Indeed, for any $s$ such that $Du(s)\in L^p$, by Poincaré inequality we have 
$$
\left\| u(s,\cdot)- \mu_s\right\|_{L^{d/(d-1)}(Q_{c_1 h})}  \leq C \int_{Q_{c_1 h}} |Du(s)|.
$$
Hence 
$$
\xi \leq C c_1^{p(d-1)}\tau \left( \fint_{Q_{c_1 h}} |Du(s)|\right)^p.
$$
\end{proof}

A first consequence of Lemma \ref{lem:pointcle} is a bound from below of $\tau$: 

\begin{Lemma} \label{lem:boundbelowtau} We have 
\be\label{boundbelowtau}
\tau \geq 3\sigma h.
\ee
In particular, 
\be\label{boundbelowtaubis}
t\leq -\sigma h /2\leq \sigma h /2\leq t+ \tau/2  
\ee
and 
\be\label{boundbelowtauter} 
t+\tau -\sigma h /2\geq \tau/2.
\ee
\end{Lemma}

\begin{proof} Recall that $t$, chosen as in Lemma \ref{lem:choixdet},  is a Lebesgue point of $s\mapsto \int_{Q_{c_1 h}} |Du(s)|^p$.  Using Jensen inequality and the first inequality in \eqref{choixdet}, we have 
$$
\left( \fint_{Q_{c_1 h}} |Du(t)|\right)^p \leq  \fint_{Q_{c_1 h}} |Du(t)|^p\leq Cc_2^{d+1} \lambda^p. 
$$
Combining \eqref{ineq:htauxibis} with \eqref{pointcle} then gives
$$
(c_1 h)^q \tau^{1-q}\leq C  \xi \leq Cc_1^{p(d-1)}  \tau \left( \fint_{Q_{c_1 h}} |Du(t)|\right)^p \leq  Cc_1^{p(d-1)} c_2^{d+1}  \tau \lambda^p.
$$
So 
$$
\tau\geq C^{-1}  c_1^{1- p(d-1)/q} c_2^{-(d+1)/q} \lambda^{1-p} h =  C^{-1}  c_1^{1- p(d-1)/q} c_2^{-(d+1)/q} \kappa^{-1}\sigma h \geq 3 \sigma h
$$
for $\kappa$ sufficiently small (depending on $c_1, c_2$). 
\end{proof}

%
%

By familiar argument we can also bound below $\xi$ by the variation in time of $u$: 

\begin{Lemma}\label{lem:intutauu0bis} We have 
\be\label{intutauu0bis}
\fint_{Q_h} (u(\sigma h/2,y)-u(-\sigma h/2,y))dy \leq  C \xi.
\ee
\end{Lemma}

\begin{proof}
Since, by \eqref{boundbelowtauter}, $\ds t+\tau -\sigma h /2\geq \tau/2$, we have by Lemma \ref{lem:subsol}, for any $y\in Q_h$, 
$$
\begin{array}{rl}
\ds u(\sigma h/2,y) \leq & \ds u(t+\tau, \gamma(t+\tau)) + C \frac{(c_1 h)^q}{(\tau/2)^{q-1}} +  C\tau\left(\fint_{(\sigma h/2,t+\tau)\times Q_{c_1 h}} f^{r_1}\right)^{1/r_1} \\
\ds \leq & \ds  u(t, x)-C^{-1}\xi  + C \frac{(c_1 h)^q}{\tau^{q-1}} + C\tau\left(\fint_{(\sigma h/2,t+\tau)\times Q_{c_1 h}} f^{r_1}\right)^{1/r_1}\; \leq \; \ds \fint_{Q_h}u(t)  + C \xi, 
\end{array}
$$
thanks again to \rife{back}, \eqref{cas2bisbis} and \eqref{ineq:htauxibis}. Next we  estimate $\ds \fint_{Q_h}u(t)$ by $\ds \fint_{Q_h}u(-\sigma h/2)$. Integrating \eqref{ineq:subsol} over $(t, -\sigma h/2)\times Q_h$ and using again \eqref{cas2bisbis}, we get
$$
\begin{array}{rl}
\ds \fint_{Q_h}u(t) \; \leq & \ds  \fint_{Q_h}u(-\sigma h/2) +  C\int_t^{-\sigma h/2} \fint_{Q_h} f \\ 
\leq & \ds   \fint_{Q_h}u(-\sigma h/2) +  Cc_1^{d/r_1} \tau \left(\fint_{(t,t+\tau)\times Q_h} f^{r_1}\right)^{1/r_1} \; \leq \;   \fint_{Q_h}u(-\sigma h/2) + C\xi. 
\end{array}
$$
So, for any $y\in Q_h$, 
$$
u(\sigma h/2,y) \leq \fint_{Q_h}u(-\sigma h/2) + C\xi. 
$$
Integrating over $y\in Q_h$ gives  \eqref{intutauu0bis}. 
\end{proof}

We are now ready to prove our main estimate: 

\begin{proof}[Proof of Proposition \ref{lem:key}.]
Combining Lemmata  \ref{lem:intutauu0bis}, \ref{lem:pointcle} and \ref{lem:cas2bis}, we have
$$
\frac{1}{\sigma h}\fint_{Q_h} (u(\sigma h/2,y)-u(-\sigma h/2,y))dy  \leq  C \xi (\sigma h)^{-1} \leq  Cc_2 c_1^{p(d-1)} \left( \fint_{Q_{c_1 h}} |Du(s)|\right)^p
$$
for a.e. $s\in [t, t+\tau/2]$. Since by \eqref{boundbelowtaubis}, $t\leq -\sigma h /2\leq \sigma h /2\leq t+ \tau/2$, we get 
$$
\begin{array}{rl}
\ds \frac{1}{\sigma h}\fint_{Q_h} (u(\sigma h/2,y)-u(-\sigma h/2,y))dy   \;  \leq  & \ds Cc_2 c_1^{p(d-1)}  \essinf_{s\in (-\sigma h/2,\sigma h/2)} \left(\fint_{Q_{c_1 h}} |Du(s)|\right)^p\\
\leq & \ds  Cc_2 c_1^{pd}  \left(\fint_{Q'} |Du|\right)^p
\end{array}
$$
Integrating inequality \eqref{ineq:subsol} on $Q$ gives
$$
\begin{array}{rl}
\ds \fint_{Q} |Du|^p \; \leq & \ds C \fint_{Q_h} \frac{u(\sigma h/2)- u(-\sigma h/2)}{\sigma h} + C \fint_{Q} f\\ 
\leq & \ds  
Cc_2 c_1^{pd}  \left(\fint_{Q'} |Du|\right)^p  + Cc_1^{d+1}\fint_{Q'} f^{r_1} .
\end{array}
$$
Recalling \eqref{keyhyp} and \eqref{casetotreat} finally gives \eqref{keycl}.
\end{proof}

\section{Proof of Theorem \ref{th:main}}

Throughout this part $u: [-1/2,1/2]\times Q_1\to \R$ is a continuous map which satisfies  \eqref{ineq:subsol} and  \eqref{ineq:supersol} in  $[-1/2,1/2]\times Q_1$.  In view of Lemma \ref{lem:distrib}, we already know that $u$ is in BV and $Du$ is in $L^p$. Our aim is now to show the higher integrability of $Du$, the fact that the measure  $\partial_tu$ is absolutely continuous and in $L^{1+\vep}$ for some $\vep>0$. Recall that $p>1$ and  $r>1+d/p$ are given and that $q$ is the conjugate exponent of $p$: $1/p+1/q=1$. 

Let us start with the  higher integrability of $Du$: 

\begin{Proposition}\label{prop:Dup+ep} There exists $\vep_0>0$ depending only on $d$, $p$, $r$ and $\bar C$ and a constant $M$, depending on $d$, $p$, $r$ and $\bar C$, $\|u\|_\infty$ and $\|f\|_r$, such that 
$$
\int_{Q_{1/2,1/2}}|Du|^{p(1+\vep_0)}\leq M.
$$
\end{Proposition}

The proof of the Proposition follows from Proposition  \ref{lem:key} and arguments developed by Kinnunen and Lewis of \cite{KiLe}. Actually, we just reproduce here---for the sake of completeness---the proof of Proposition 4.1 of \cite{KiLe} which explains why Proposition \ref{lem:key} implies the higher integrability of $Du$.

\begin{proof} To simplify the notation, we assume as before that $f\geq 1$. Fix $r_1\in (1+d/p,r)$ and let  $\lambda_0$, $\kappa$, $c_1$, $c_2$ and $\hat C$ be as in Proposition \ref{lem:key}. For  $\lambda\geq \lambda_0$ we set $\sigma= \kappa \lambda^{1-p}$ and define 
$$
E(\lambda)=\{ (t,x)\in Q_{1/2,1/2}\;, \; |Du(t,x)|^p+f^{r_1}(t,x)>\lambda^p\}.
$$ 
We can (and will) assume without loss of generality that $\sigma<1$. 
Let $(t,x)\in E(\lambda)$ be a Lebesgue point of $|Du|^p+ f^{r_1}$. We first claim that there exists $h_{t,x}\in (0, 1/(4c_2))$ such that 
\be\label{hbreslbkj}
\lambda^p = \fint_{Q_{\sigma h_{t,x}, h_{t,x}}(t,x)}(|Du|^p+ f^{r_1}) \leq c_2^{d+1}  \fint_{Q_{ c_2 \sigma h_{t,x},c_2  h_{t,x}}(t,x)}(|Du|^p+ f^{r_1})\leq c_2^{d+1} \lambda^p.
\ee
Indeed, we note that, by Lebesgue Differentiation Theorem, 
$$
\lim_{h\to 0^+} \fint_{Q_{\sigma h, h}(t,x)}(|Du|^p+ f^{r_1})= |Du(t,x)|^p+f^{r_1}(t,x)>\lambda^p. 
$$
On the other hand, 
$$
\begin{array}{rl}
\ds \fint_{Q_{\sigma /(4c_2), 1 /(4c_2)}(t,x)}(|Du|^p+ f^{r_1}) \;  \leq & \ds Cc_2^{d+1}\sigma^{-1} \int_{Q_{1,1}}(|Du|^p+ f^{r_1}) \\
\leq& \ds Cc_2^{d+1}\kappa^{-1}\lambda^{p-1} \int_{Q_{1,1}}(|Du|^p+ f^{r_1})\; <\;  \lambda^p, 
\end{array}
$$
since $\lambda\geq \lambda_0$ and $\lambda_0$ is large enough, depending on  $\|Du\|_p$ and  $\|f\|_r$. As $\ds h\to \fint_{Q_{\sigma h,  h}}(|Du|^p+ f^{r_1})$ is continuous, there is a largest real number  $h_{t,x} \in (0,1/(4c_2))$ for which  equality 
$\ds
\lambda^p = \fint_{Q_{\sigma h_{t,x}, h_{t,x}}}(|Du|^p+ f^{r_1})
$
holds.  Since $c_2 h_{t,x}$ is a larger value than $h_{t,x}$, this in particular implies the  right-hand side inequality in \eqref{hbreslbkj}. \\

By  Proposition \ref{lem:key} and the previous argument, we obtain that,  for almost every $(t,x)\in E(\lambda)$ there exists $h_{t,x}\in (0,1/2)$ such that \eqref{hbreslbkj} holds, which implies that
\be\label{khbjhil}
\fint_{Q_{c_2\sigma  h_{t,x}, c_2  h_{t,x}}}|Du|^p \leq \hat C \left(\fint_{Q_{c_1 \sigma h_{t,x}, c_1   h_{t,x}}} |Du|\right)^p+  \hat C \fint_{Q_{c_1\sigma  h_{t,x}, c_1   h_{t,x}}} f^{r_1}. 
\ee
Since $5c_1\leq c_2$, Vitali Covering Theorem yields the existence of an enumerable family of cubes $(Q_i'':= Q_{c_2 \sigma  h_{t_i,x_i}, c_2  h_{t_i,x_i}})$ such that the $(Q_i':= Q_{c_1 \sigma h_{t_i,x_i}, c_1 h_{t_i,x_i}})$ have an empty intersection and the $(Q''_i)$ cover $E(\lambda)$. 
 
Let us set $g= (|Du|^p+f^{r_1})^{1/p}$.  We denote henceforth by $C$ any generic constant, depending only on $d$, $p$, $r$ and $\bar C$, and in particular independent of $\la$. Combining \eqref{hbreslbkj} and \eqref{khbjhil} we get 
\be\label{lkhjdQCLS}
\ds C^{-1}\lambda^p \; \leq  \ds  \fint_{Q''_i}g^p  \leq C \left(\fint_{Q'_i} g\right)^p+ C \fint_{Q_i'} f^{r_1}\\
 \leq  \ds C^2\fint_{Q_i''}g^p \leq C^3\lambda^p. 
\ee
For $\eta\in (0,1)$ small to be chosen later, we have 
$$
\begin{array}{rl}
\ds \left(\fint_{Q'_i} g\right)^p  \; \leq & \ds C\eta^p\lambda^p + C \left(|Q'_i|^{-1}\int_{Q'_i\cap E(\eta \lambda)} g\right)^p\\
\leq & \ds 
C\eta^p\lambda^p+  C\left(\fint_{Q'_i} g^p\right)^{(p-1)/p}|Q'_i|^{-1}\int_{Q'_i\cap E(\eta \lambda)} g \\ 
\leq &\ds  C\eta^p\lambda^p+  C\lambda^{p-1}|Q'_i|^{-1}\int_{Q'_i\cap E(\eta \lambda)} g
\end{array}
$$
while, in the same way,  
$$
\fint_{Q_i'}f^{r_1} \leq C\eta^p\lambda^p + |Q'_i|^{-1} \int_{Q'_i\cap E(\eta \lambda)} f^{r_1}.
$$
We can combine the above inequalities and obtain  by \eqref{lkhjdQCLS} 
$$
C^{-1}\lambda^p\leq  \fint_{Q''_i}g^p\leq C\eta^p\lambda^p+ C |Q'_i|^{-1}\int_{Q'_i\cap E(\eta \lambda)} (\lambda^{p-1} g + f^{r_1}).
 $$
For  $\eta>0$ small enough (depending only on $d$, $p$, $r$ and $\bar C$) one can absorb the term $C\eta^p\lambda^p$ into the left-hand side and get
 $$
 \int_{Q''_i}g^p\leq C \int_{Q'_i\cap E(\eta \lambda)} (\lambda^{p-1} g + f^{r_1}).
 $$
Since $\ds E(\lambda)\subset \bigcup_i Q_i''$ and the $(Q_i')$ have an empty intersection, we obtain by summing up over $i$: 
\be\label{ifhzoekhb}
 \int_{E(\lambda)}g^p\leq C \int_{E(\eta \lambda)} (\lambda^{p-1} g + f^{r_1})\qquad \forall \lambda\geq \lambda_0.
 \ee
By standard argument, which can be found in  \cite{KiLe}, one deduces the existence of $\vep_0$ and $M$ such that 
$$
\int_{Q_{1/2,1/2}} g^{p(1+\vep_0)} \leq M. 
$$ 
Since $r_1<r$, the higher integrability for $g$ implies the higher integrability for the single term $|Du|^p$.
Note that the improved integrability exponent $p(1+\vep_0)$ depends on the constants $C$, $\eta$ and $p$ in inequality \eqref{ifhzoekhb}, hence only on $d$, $p$, $r$ and $\bar C$. On the other hand, the norm $\|Du\|_{p(1+\vep_0)}$ also depends on $\lambda_0$ and $\|f\|_{r_1}$, i.e., on $d$, $p$, $r$ and $\bar C$, $\|Du\|_p$ and $\|f\|_r$. Integrating inequality \eqref{ineq:subsol} over $Q_{1,1}$ shows that $\|Du\|_p$ is bounded above by $\|u\|_\infty$ and  $\|f\|_r$, whence the conclusion. 
\end{proof}

Next we show  that $u\in W^{1,1}_{loc}$ and that $\partial_tu$ belongs to $L^{1+\vep}_{loc}$ for $\vep>0$ small enough. Our starting point is a weak integral form of \eqref{ineq:supersol}: 

\begin{Lemma}\label{lem:estiu-u} Fix $r_1\in (1+d/p,r)$. There exists $\bar c, C\geq 2$, depending only on $d$, $p$, $r_1$, $r$ and $\bar C$,  such that, for any $(t,x)\in Q_{1/2,1/2}$ and any $h\in (0, 1/(2 \bar c))$,
\be\label{tictac}
\partial_tu(Q_{h,h}(t,x)) \leq C\int_{Q_{\bar c h,\bar c h}(t,x)} (|Du|^p+ 1+\max_{Q_{\bar c h,\bar c h}(t,x)} f^{r_1}).
\ee
\end{Lemma} 

 In particular, by continuity of $f$, there exists $h_0>0$ such that, for any $(t,x)\in Q_{1/2,1/2}$ and any $h\in (0, h_0)$,
\be\label{tictacbis}
\partial_tu(Q_{h,h}(t,x)) \leq C\int_{Q_{\bar c h,\bar c h}(t,x)} (|Du|^p+ 1+ f^{r_1}),
\ee
where the constants $\bar c$ and $C$ depend only on $d$, $p$, $r_1$, $r$ and $\bar C$.

Note that if we knew that \eqref{ineq:supersol} holds in the sense of distributions, then the result would be obvious (and actually much sharper). We prove in the next section that this is the case by showing that $u$ is differentiable a.e. It would be interesting to have a direct proof of this fact, but we are not aware of any result in this direction.

\begin{proof} To simplify the notation, we assume that $f\geq 1$. The proof is mostly a variation on Lemmata \ref{lem:pointcle} and \ref{lem:intutauu0bis}. The main step is the following: if $(t,x)\in Q_1$ and $h\in (0, 1/(2 \bar c))$, then there exists $\tau\in(C^{-1}h^p,h)$ such that
\be\label{estiu-utau} 
\fint_{Q_h(x)} (u(t+\tau/2,y)-u(t,y))dy  \leq C\tau  \left(\fint_{(t,t+\tau/2)\times Q_{\bar c h}(x)} |Du|\right)^p+   C\tau \max_{[t,t+h]\times Q_{\bar c h}(x)} f^{r_1}. 
\ee
Let us complete the proof of the Lemma before showing \eqref{estiu-utau}. Using Hölder inequality, we have from \eqref{estiu-utau} that, for any $t\in (-h/2,h/2)$, there exists $\tau\in (C^{-1}h^p,h)$ such that
\be\label{estiu-utau+}
\fint_{Q_h(x)} (u(t+\tau/2,y)-u(t,y))dy  \leq C \int_t^{t+\tau/2} \fint_{Q_{\bar c h}(t,x)} |Du|^p+ C \tau \max_{[t,t+h]\times Q_{\bar c h}(x)} f^{r_1}. 
\ee
We define inductively the sequence of times $(t_i)$ by $t_0= -h/2$, $t_{i+1}=t_i+ \tau_i/2$ where $\tau_i$ is associated with $(t_i,x)$ as in \eqref{estiu-utau+}. Then there exists $I$ such that $t_{I}\leq h/2< t_{I+1}$ and we have 
$$
\fint_{Q_h(x)} (u(t_{I+1},y)-u(-h/2,y))dy  \leq C \int_{-h/2}^{t_{I+1}} \fint_{Q_{\bar c h}(t,x)} |Du|^p+ C (t_{I+1}+h/2) \max_{[-h/2,2h]\times Q_{\bar c h}(x)} f^{r_1}.  
$$
As $t_{I+1}-t_I\leq  h/2$, we obtain 
$$
\fint_{Q_h(x)} (u(h/2,y)-u(-h/2,y))dy  \leq C \int_{-h}^{ h} \fint_{Q_{\bar c h}(t,x)} |Du|^p+ C h \max_{[-\bar ch,\bar ch]\times Q_{\bar c h}(x)} f^{r_1}, 
$$
which implies \eqref{tictac}. \\

We now prove \eqref{estiu-utau}. To fix the ideas, we assume $(t,x)=(0,0)$. For $\bar c\geq 2$ to be defined below we set $\ds C_0:= \max_{[0,h]\times Q_{\bar c h} } f$. Let $z\in Q_h$ be such that $\ds \fint_{Q_h}u(0,y)dy= u(0,z)$ and $\gamma$ be a generalized characteristic for $u(0,z)$. Let also $\tau$ be the largest time such that $\tau\leq \theta (1+C_0)^{-1/q}\bar ch$ (where $0<\theta<1$ is a small constant to be chosen below) and $\gamma([0,\tau])\subset Q_{\bar c h/2}$. 

If  $\tau= \theta (1+C_0)^{-1/q}\bar ch$, then, as $\gamma$ is a generalized characteristic,  
$$
\begin{array}{rl}
\ds u( \tau/2,y) \leq & \ds u(\tau, \gamma(\tau)) + C \frac{(\bar c h)^q}{(\tau/2)^{q-1}} +  C_0 \tau \\
\ds \leq & \ds  u(0, z) + C \frac{(\bar c h)^q}{\tau^{q-1}} + (C+C_0)\tau\; \leq \;  u(0, z) +  C \theta^{-q}(1+C_0)\tau.
\end{array}
$$
Integrating over $Q_h$ and using the definition of $z$: 
$$
\fint_{Q_h} (u( \tau/2,y)-u(0,y))dy \leq C \theta^{-q}(1+C_0)\tau.
$$
We now assume that $\tau< \theta (1+C_0)^{-1/q}\bar ch$, so that $|\gamma(\tau)|=\bar c h/2$. We set $\ds \xi= \int_0^\tau |\dot \gamma|^q$. 
One can check, exactly as in Lemma \ref{lemineq:htauxibis}, that 
\be\label{ljvsdjkn}
\frac{1}{C}\frac{\bar c^q  h^q}{\tau^{q-1}}\leq \xi \leq C \frac{\bar c^q  h^q}{\tau^{q-1}}.
\ee
Following the proof of Lemma \ref{lem:utyBbis}, we have for any $(s,y)\in  [0,\tau/2] \times B(\gamma(\tau), h)$, 
$$
\begin{array}{rl}
u(s,y) \leq & \ds u(\tau, \gamma(\tau)) + C \frac{h^q}{\tau^{q-1}} + C_0\tau  \\
\leq & \ds u(0,z)-\frac{1}{C} \xi + C \bar c^{-q}\xi  + (C+C_0)\tau
\end{array}
$$
where, by \eqref{ljvsdjkn} and since  $\tau< \theta (1+C_0)^{-1/q}\bar ch$, 
$$
(C+C_0)\tau\leq C(1+C_0)\tau \leq C(\theta\bar c h \tau^{-1})^q\tau \leq C \theta^q \xi. 
$$
This implies that, for $\bar c$ large and $\theta$ small,  there exists $C_1\geq 1$, depending only on $d$, $p$, $r_1$, $r$ and $\bar C$, such that
\be\label{kjhabshbb}
 u(s,y)\leq u(0,z)-C_1^{-1} \xi \qquad \forall  (s,y)\in [0,\tau/2] \times B(\gamma(\tau), h).
\ee
Then the same argument as in Lemma \ref{lem:pointcle} implies that 
$$
\xi\tau^{-1}\leq C\left(\fint_{Q_{\bar c h}} |Du(s)|\right)^p \qquad \mbox{\rm for a.e. } s\in (0,\tau/2)
$$
while, if we argue as in Lemma \ref{lem:intutauu0bis}, we get 
$$
\fint_{Q_h} (u(\tau/2,y)-u(0,y))dy \leq  C \xi.
$$
To summarize, we just proved that, if $\tau< \theta (1+C_0)^{-1/q}\bar ch$, then 
$$
\fint_{Q_h} (u(\tau/2,y)-u(0,y))dy \leq  C\tau \left(\fint_{(0,\tau/2)\times Q_{\bar c h}} |Du(s)|\right)^p,
$$
while if  $\tau= \theta (1+C_0)^{-1/q}\bar ch$, then 
$$
\fint_{Q_h} u( \tau/2,y)-u(0,y)dy \leq C \theta^{-q}(1+C_0)\tau.
$$
Putting the two cases together yields \eqref{estiu-utau}. 

To check that $\tau\geq C^{-1} h^p$, we note that \eqref{kjhabshbb} implies that $\xi\leq C \max_{ Q_1}|u|$ so that \eqref{ljvsdjkn} entails that
$\tau \geq C^{-1}  \left(\max_{ Q_1}|u|\right)^{1-q} (\bar c h)^p$. 
\end{proof}

As a consequence of the Lemma we have: 

\begin{Corollary}\label{cor:Dtu+ep} The map $u$ belongs to $W^{1,1}(Q_{1/2,1/2})$ and 
$$
\int_{Q_{1/2,1/2}} |\partial_t u|^{1+\vep_0} \leq M,
$$
where $\vep_0$ is the constant defined in Proposition \ref{prop:Dup+ep} and $M$ is a constant depending on $d$, $p$, $r$, $\bar C$, $\|u\|_\infty$ and $\|f\|_r$. 
\end{Corollary}

The derivation of the absolute continuity of $\partial_tu$ from Lemma \ref{lem:estiu-u} is standard. We give the proof for sake of completeness. 

\begin{proof} Fix $r_1\in (1+d/p,r)$.
From Lemma \ref{lem:distrib} we already know that $u$ is in BV. Moreover, as inequality \eqref{ineq:subsol} holds in the sense of distributions, the singular part $(\partial_t u)^s$ of the measure $\partial_t u$ is nonnegative.
In order to check the absolute continuity of the measure $\partial_tu$, we now show that 
\be\label{bounddtu}
 \partial_tu(A) \leq C \int_A \left( |Du|^p+1 +  f^{r_1}\right)
 \ee
 for any Borel subset $A$ of $Q_{1/2,1/2}$. As $\partial_tu$ is a Borel measure, we just need to prove this when $A$ is closed. 
 Set $g:= C( |Du|^p+1 +  f^{r_1})$. Fix $\vep>0$ small and, for any $h\in(0,\vep)$ and $(t,x)\in Q_{1/2,1/2}$, we claim that there exists $\rho\in(0,h)$  such that 
\be\label{choixrho}
\fint_{Q_{\bar c \rho, \bar c  \rho}(t,x)} g\leq (1+\vep) \fint_{Q_{\rho,\rho}(t,x)} (g+\vep)
\ee
where $\bar c $ is defined in Lemma \ref{lem:estiu-u}. Indeed, otherwise, one has in particular
$$
\fint_{Q_{\bar c \rho, \bar c  \rho}(t,x)} g> (1+\vep) \fint_{Q_{\rho,\rho}(t,x)} g\qquad \forall \rho\in (0,h).
$$
One then deduces by induction that, for any $n\in \N$,  
$$
\max_{\rho\in [\bar c^{-n-1}h, \bar c^{-n}h]} \fint_{Q_{\rho,\rho}(t,x)} g \leq (1+\vep)^{-n} \max_{\rho\in [h,\bar c^h]} \fint_{Q_{\rho,\rho}(t,x)} g. 
$$
So 
$$
\lim_{\rho\to 0} \fint_{Q_{\rho,\rho}(t,x)} g=0,
$$
which is impossible since $g\geq 1$. Whence the existence of $\rho$ as in \eqref{choixrho}. 

We denote by ${\mathcal F}$ the collection of the cubes $Q_{\rho,\rho}(t,x)$ as $(t,x)\in A$ and $h$ and $\rho$ are  as above ($\vep$ being fixed). Vitali Lemma then says that we can find a disjoint family ${\mathcal F}'\subset {\mathcal F}$ such that 
$$
\partial_tu\left(A\backslash \bigcup_{{\mathcal F}'} Q_{\rho,\rho}(t,x)\right)=0.$$ 
Then, by Lemma \ref{lem:estiu-u} (see also \eqref{tictacbis}) and \eqref{choixrho},  
$$
\begin{array}{rl}
\ds \partial_tu(A) \; \leq & \ds  \partial_tu(\bigcup_{{\mathcal F}'} Q_{\rho,\rho}(t,x)) \leq  \ds \sum_{{\mathcal F}'} \partial_tu(Q_{\rho,\rho}(t,x)) \\ 
\leq & \ds  \sum_{{\mathcal F}'} \int_{Q_{\bar c \rho,\bar c \rho}(t,x)} g \; \leq \;  (1+\vep) \sum_{{\mathcal F}'} \int_{Q_{\rho, \rho}(t,x)} (g +\vep) \; 
\leq \; (1+\vep) \int_{\bigcup_{{\mathcal F}'}Q_{\rho, \rho}(t,x)} (g +\vep). 
 \end{array}
$$
As $\rho<\vep$, we have 
$$
\partial_tu(A) \; \leq\; (1+\vep) \int_{\bigcup_{{\mathcal F}'}Q_{\rho, \rho}(t,x)} (g +\vep) \leq
(1+\vep) \int_{A+ \vep B(0,1)} (g+\vep).
 $$
Letting $\vep\to0$ yields \eqref{bounddtu}. 

In particular $\partial_t u$ is absolutely continuous and we can readily derive from \eqref{bounddtu} the upper bound 
$$
 \partial_tu\; \leq \; 
\ds C \left( |Du|^p+1
 +  f^{r_1}\right)\qquad {\rm a.e.}
$$ 
The lower bound 
$$
\ds  \partial_tu\; \geq \; \frac{1}{\bar C} |Du|^p - f\qquad {\rm a.e.}
 $$
 is given by \eqref{ineq:subsol} since it holds in the sense of distributions (Lemma \ref{lem:distrib}) and therefore a.e. since any term in the inequality is absolutely continuous. As, by Proposition \ref{prop:Dup+ep},  $|Du|^p$ and $f^{r_1}$ belong to $L^{1+\vep}$ for $\vep$ small enough, so does $\partial_t u$. 
\end{proof}

\begin{proof}[Proof of Theorem \ref{th:main}.] It is a straightforward consequence of Proposition \ref{prop:Dup+ep} and Corollary \ref{cor:Dtu+ep} after suitable scaling.  
\end{proof}

\section{A.e. differentiability}\label{sec:aediff}

In this section we prove the a.e. differentiability of a map $u$ satisfying \eqref{ineq:subsol} and \eqref{ineq:supersol}. We actually show a slightly stronger result, which will be used in  section \ref{sec:MFG}: in the inequality \eqref{ineq:subsol} we no longer require the continuity of right-hand side $f$, but only assume that  $f$ belongs to $L^r((0,1)\times Q_1)$. This requires a change of meaning for \eqref{ineq:subsol}, which is now supposed to hold in the sense of distribution. As in the previous sections, $p>1$ and  $r>1+d/p$ are given and $q, r'$ are the conjugate exponent of $p, r$ respectively. 

\begin{Proposition}\label{AeDiff} Let $u\in W^{1,1}(Q_{1,1})\cap C^0(\overline{Q_{1,1}})$ be such that $Du\in L^p(Q_{1,1})$. We assume that $u$ satisfies \eqref{ineq:subsol} in the sense of distributions and \eqref{ineq:supersol} in the viscosity sense. Then $u$ is differentiable at almost every point of $Q_{1,1}$. 
\end{Proposition}

Note that the above proposition implies the last statement of Theorem \ref{th:main}: indeed, if $f$ is continuous, then a subsolution of \eqref{ineq:subsol} in the viscosity sense is also a subsolution in the sense of distributions (Lemma \ref{lem:distrib}) and therefore Proposition \ref{AeDiff} applies. It also shows that Lemma \ref{lem:estiu-u} only gave a very rough estimate of $\partial_t u$ and that inequality \eqref{ineq:supersol} actually also holds a.e. \\

The proof of Proposition \ref{AeDiff} requires two preliminary results. The first one is a variant of Lemma \ref{lem:subsol}. 

\begin{Lemma} Let  $f\in L^r(Q_{1,1})$, $u\in C^0(\overline{Q_{1,1}})$ be a subsolution of \eqref{ineq:subsol} in the sense of distributions. Then for any $(t,x),(s,y)\in Q_{1/2,1/2}$ with $t<s$ one has
\be\label{blabla}
u(t,x)\leq u(s,y)+C\frac{|x-y|^q}{(s-t)^{q-1}}+ C(\|f\|_{L^r(Q_{1,1})}+1)(s-t)^\alpha
\ee
where $\alpha= (p(r-1)-d)/(p(r+1)-1)$.
\end{Lemma}

Note that, by our assumption $r>1+d/p$, $\alpha$ is positive. 

\begin{proof} We assume that $u$ and $f$ are smooth, the general case being directly obtained by convolution. We follow the proof of Lemma \ref{lem:subsol} (with $r_1=r$), choosing now $\delta=h^{\alpha_1}$, with $$\alpha_1=-(\beta d+1+qr(\beta-1))/(qr+d),$$ and $\beta\in (1/p, (r-1)/d)$. We get 
$$
\begin{array}{rl}
\ds u(t,x)\; \leq & \ds u(s,y) +C\frac{|x-y|^q}{(s-t)^{q-1}}+ C\delta^q(s-t)^{q(\beta-1)+1}+ C(\|f\|_r+1)\delta^{-d/r}(s-t)^{-\beta d/r+ 1/r'}\\
\leq & \ds u(s,y) +C\frac{|x-y|^q}{(s-t)^{q-1}}+ C(\|f\|_r+1) (s-t)^\alpha.
\end{array}
$$ 
\end{proof}

The next ingredient for the proof  of Proposition \ref{AeDiff} is a Hölder regularity result for a map satisfying \eqref{ineq:subsol} and \eqref{ineq:supersol}. Such a result has first been proved in \cite{CS} in a slightly different context.

\begin{Proposition}\label{ReguHol} Let $u\in C^0(\overline{Q_{1,1}})$ be such that $Du\in L^p(Q_{1,1})$. Assume that $u$ satisfies \eqref{ineq:subsol} in the sense of distributions and \eqref{ineq:supersol} in the viscosity sense. Then there exists $\theta\in (0,1)$ and $M>0$, depending only on $\|u\|_\infty$, $\|f\|_r$, $d$, $r$, $\bar C$ and $p$ such that 
$$
|u(t,x)-u(s,y)|\leq M(|x-y|^\theta+|s-t|^\theta) \qquad \forall (t,x),(s,y)\in Q_{1/2,1/2}.
$$
\end{Proposition}

\begin{proof} The main idea of the proof consists in showing that generalized characteristics enjoy more regularity than the natural bound $\|\dot \gamma\|_q\leq C$ would let us expect. 

Let $\gamma$ be a generalized characteristic starting from a point $(t,x)\in  Q_{1/2,1/2}$. We have, in view of \eqref{blabla} and the definition of generalized characteristics: for any $h\in (0,1/2)$,
$$
\begin{array}{rl}
\ds u(t,x) \; \leq & \ds u(t+h,\gamma(t+h))+ C\frac{|\gamma(t+h)-x|^q}{h^{q-1}}+ C(\|f\|_r+1)h^\alpha \\
\leq & \ds u(t,x)-\int_t^{t+h} \frac{1}{C}|\dot \gamma|^q + C\frac{|\gamma(t+h)-x|^q}{h^{q-1}}+ C(\|f\|_r+1)h^\alpha\;.
\end{array}$$
Rearranging, we obtain therefore a {\it weak reverse inequality} for $|\dot \gamma|$: 
$$
\fint_t^{t+h} |\dot \gamma|^q \leq A \left[ \fint_t^{t+h} |\dot \gamma|\right]^q + B h^{\alpha-1}\qquad \forall h\in (0,1/2).
$$
where $A$ and $B$ depend only on $\|f\|_p$, $d$, $p$, $r$ and $\bar C$. 
Then Lemma 3.4 of \cite{CC} yields the existence of $\theta>q$, depending only on $\|f\|_p$, $d$, $p$, $r$, $\|u\|_\infty$ and $\bar C$ such that 
\be\label{qnrvsljkn}
\int_t^{t+h} |\dot \gamma| \leq Ch^{1-1/\theta}\qquad \forall h\in [0,1/2]
\ee
(actually Lemma 3.4 in \cite{CC} is stated for stochastic processes under extra conditions on $p$ and $r$; a careful inspection of the proof shows that these conditions are unnecessary for deterministic processes). One can then derive from \eqref{qnrvsljkn} Hölder estimates for the solution $u$ exactly as in the proof of Theorem 4.1 of \cite{CC}. 
\end{proof}

\begin{proof}[Proof of Proposition \ref{AeDiff}.] Let $(t,x)$ be a point of approximate differentiability of $u$ and a Lebesgue point of $f$. Recall that, as $u$ belongs to $W^{1,1}$, $u$ is almost everywhere approximately differentiable. For $\rho>0$ small, let 
$$
u_\rho(s,y)=\rho^{-1}\left[ u(t+\rho s, x+\rho y)- u(t,x)\right]\qquad  {\rm in }\; (-2,2)\times Q_2.
$$
Since $u$ is approximately differentiable at $(t,x)$, $u_\rho$ converge to $g$ in $L^1_{loc}$ where $g$ is the affine function
$$g(s,y):= \partial_tu(t,x)s+\lg Du(t,x), y\rg.
$$
 In order to show that $u$ is differentiable at $(t,x)$, we just need to check that this convergence is uniform. Note that $u_\rho$ satisfies the two inequalities:
\be\label{ineq:subsol3}
-\partial_t u_\rho +\frac{1}{\bar C} |Du_\rho|^p \leq f_\rho\qquad  {\rm in }\; (-2,2)\times Q_2
\ee
in the sense of distributions, where $f_\rho (s,y):=f((t,x)+\rho(s,y))$, and 
\be\label{ineq:supersol3}
-\partial_t u_\rho +\bar C |Du_\rho|^p \geq -\bar C\qquad  {\rm in }\; (-2,2)\times Q_2
\ee
in the viscosity sense. 

As $(t,x)$ is a Lebesgue point of $f$, $f_\rho$ converges to $f(t,x)$ in $L^r_{loc}$. In particular, $(\|f_\rho\|_{L^r((-2,2)\times Q_2)})$ is bounded. Using \eqref{blabla} we infer that, for any $(s_1,y_1), (s_2,y_2)\in (-1,1)\times Q_1$ with $s_2>s_1$, 
\be\label{tototo}
u_\rho(s_1,y_1)\leq u_\rho(s_2,y_2)+C\frac{|y_2-y_1|^q}{(s_2-s_1)^{q-1}}+ C(s_2-s_1)^\alpha.
\ee
As $u_\rho$ converges in $L^1_{loc}$ and \eqref{tototo} holds, one easily concludes that $u_\rho$ is locally uniformly bounded in $(-1,1)\times Q_1$. Proposition \ref{ReguHol} then says that the $u_\rho$ are locally uniformly Hölder continuous. Again by $L^1$ convergence, this implies that the $u_\rho$ converge in fact locally uniformly to $g$. Thus $u$ is differentiable at $(t,x)$. 
\end{proof}

\section{Examples}\label{sec:examples}

In this section, we discuss the sharpness of the assumptions and of the conclusion of Theorem \ref{th:main}. We first show through an example that a map satisfying \eqref{ineq:subsol} and \eqref{ineq:supersol} does not necessarily belong to $W^{1,1+\vep}$ for large values of $\vep$ (Proposition \ref{prop:ex}): this means that the conclusion is somewhat sharp. Concerning the assumption, we explain in Remark \ref{remrem} that the conclusion of Theorem \ref{th:main} requires the supersolution inequality to hold {\it in the viscosity sense}, and not in an a.e. sense. 

The following  example is inspired  by \cite{AA} and \cite{CC}. Let us fix $p>1$ and $\gamma\in(1-1/q,1)$ where $q$  is as usual the conjugate exponent of $p$, i.e., $1/p+1/q=1$. We define 
\begin{equation}\label{defxi0}
\xi_0(t)=t^\gamma\qqf t\in[0,1]\,,
\end{equation}
\begin{equation*}
a(t,x)=\begin{cases}
G^{-1} & \mbox{if}\quad x=\xi_0(t)
\\
MG^{-1} & \mbox{if}\quad x\neq\xi_0(t)
\end{cases}\qqf (t,x )\in[0,1]\times \R\, ,
\end{equation*}
\begin{equation*}
g(x)=\begin{cases}
0 & \mbox{if}\quad x=1
\\
1 & \mbox{if}\quad x\neq 1
\end{cases}
\qqf x\in\R
\end{equation*}
where $M$ and $G$ are real numbers such that 
\begin{equation}\label{choixMG}
M\; >\; \frac{\gamma^q}{1-(1-\gamma)q}\qquad {\rm and }\qquad G\;>\;\frac{\gamma^q}{q(1-(1-\gamma)q)}\,.
\end{equation}
Note that $M>1$. Our aim is to show that ``the solution" $u$ to the Hamilton-Jacobi equation
\be\label{ex:u}
\begin{cases}\ds
-\partial_t u+{|Du|^p\over pa^{p-1}(t,x)}=0 &\mbox{in}\quad (0,1)\times \R
\vspace{.5mm}
\\
u(1,x)=g(x) & x\in\R
\end{cases}
\ee
is such that  $\partial_t u\notin L^{1+\vep}((0,1)\times (0,1))$ and $Du\notin L^{p(1+\vep)}((0,1)\times (0,1))$ for any
$\vep\geq \frac{\gamma(q+1)+1-q}{q(1-\gamma)}$. Since $a$ and $g$ are discontinuous, the meaning of the above equation is not clear. To overcome this issue,  we approximate $a$ and $g$ in a suitable way and show that the solutions of the approximate problem cannot be bounded in $W^{1,1+\vep}$. 

Let us now fix two sequences
\begin{equation*}
a_n:\R\times[0,1]\to\R\qmb{and}\qquad g_n:\R\to\R\qquad (n\ge 1)
\end{equation*}
of Lipschitz continuous functions such that
\begin{equation}\label{eq:a_n}
\begin{cases}
G^{-1} \le a_n(t,x)\le a(t,x)&\forall n\ge 1
\\
a_n(t,x)\uparrow a(t,x) & n\to\infty
\end{cases}\qqf (t,x)\in[0,1]\times \R
\end{equation}
and
\begin{equation}\label{eq:g_n}
\begin{cases}
0\le g_n(x)\le g(x) &\forall n\ge 1
\\
g_n(x)\uparrow g(x) & n\to\infty
\end{cases}\qqf x\in\R\,.
\end{equation}
(since $a$ and $g$ are lower semi-continuous, such approximating sequences can be built by inf-convolution). 

\begin{Proposition}\label{prop:ex} For any integer $n\ge 1$ let $a_n$ and $g_n$ satisfy \eqref{eq:a_n} and \eqref{eq:g_n}, respectively, and let $u_n$ be the Lipschitz continuous viscosity solution of 
\begin{equation}\label{defun}
\begin{cases}\ds
-\partial_t u+{|Du|^p\over p a^{p-1}_n(t,x)}=0 &\mbox{in}\quad (0,1)\times \R
\vspace{.5mm}
\\
u(1,x)=g_n(x) & x\in\R\,.
\end{cases}
\end{equation}
Then $\sup_n\|u_n\|_\infty\leq 1$ but, if $\vep\geq \frac{\gamma(q+1)+1-q}{q(1-\gamma)}$, the sequences  $(\partial_t u_n)$ and $(D u_n)$ are not bounded in $L^{1+\vep}((0,1)\times (0,1))$ and $L^{p(1+\vep)}((0,1)\times (0,1))$ respectively.
\end{Proposition}

Note that the quantity $\frac{\gamma(q+1)+1-q}{q(1-\gamma)}$ is positive, but arbitrarily small when $\gamma$ is close to $1-1/q$ and $q$ is close to $1$ (i.e., $p$ large). This shows that the $\vep$ which appears in Theorem \ref{th:main} depends on the coefficients $\bar C$ and $p$ and can be arbitrarily small.  

\begin{proof}
Let $u:[0,1]\times \R\to \R$ be the value function of the problem
\begin{equation*}\label{eq:u}
u(t,x)=\inf\left\{J[\xi, t, x]\ :\quad \xi\in W^{1,p}([t,1])\,,\;\xi(t)=x\right\}
\end{equation*}
where 
$$
J[\xi, t, x]=\int_t^1\frac{a(\xi(s),s)}{q}|\xi'(s)|^qds+g(\xi(1))\quad \forall(t,x)\in(0,1)\times \R\,,\; \xi\in W^{1,q}([t,1])\,,\;\xi(t)=x\,.
$$
We note that $u$ is---at least formally---a solution to \eqref{ex:u}. 
The main part of the proof consists in showing that $\partial_t u\notin L^{1+\vep}((0,1)\times (0,1))$ and $Du\notin L^{p(1+\vep)}((0,1)\times (0,1))$ for any $\vep\geq \frac{\gamma(q+1)+1-q}{q(1-\gamma)}$. For this we analyze optimal solutions of the problem.

We first claim that  the curve $\xi_0$ defined in \eqref{defxi0} is the unique minimizer for $u(0,0)$. For this, let $\xi$ be an optimal trajectory for $J[\cdot, 0,0]$. Assume for a while that $\xi(1)\neq \xi_0(1)$. Then let $t_1$ be the largest time such that $\xi(t)= \xi_0(t)$. By optimality of $\xi$, we have
$$
\int_{t_1}^1 \frac{a(t,\xi(t))}{q}|\xi'(t)|^qdt+g(\xi(1)) \leq \int_{t_1}^1  \frac{a(t,\xi_0(t))}{q}|\xi_0'(t)|^qdt+g(\xi_0(1)), 
$$
which implies that
$$
1\; =\; g(\xi(1))\; \leq \; \int_{t_1}^1  \frac{G^{-1}}{q}|\gamma t^{\gamma-1}|^qdt \; \leq \; G^{-1}\frac{\gamma^q}{q(1-(1-\gamma)q)}.
$$
This is impossible from the choice of $G$ in \eqref{choixMG}. 
Let us now assume that $\xi\neq \xi_0$ and let $(t_1,t_2)$ be a maximal interval on which $\xi\neq \xi_0$. As $\xi(0)=0=\xi_0(0)$ and $\xi(1)= \xi_0(1)$, we must have $\xi(t_1)=\xi_0(t_1)=t_1^\gamma$ and $\xi(t_2)=\xi_0(t_2)=t_2^\gamma$. Since  the map $(t,x)\to a(t,x)$ is constant outside the graph of $\xi_0$, $\xi$ must be a straight line between $t_1$ and $t_2$. Hence, by optimality,
$$
\begin{array}{rl}
\ds \frac{MG^{-1}}{q} \frac{(t_2^\gamma-t_1^\gamma)^q}{(t_2-t_1)^{q-1}}\; = & \ds \int_{t_1}^{t_2} \frac{a(t,\xi(t))}{q}|\xi'(t)|^qdt \\
\leq &\ds   \int_{t_1}^{t_2}  \frac{a(t,\xi_0(t))}{q}|\xi_0'(t)|^qdt \; =\; G^{-1}\frac{\gamma^q}{q}\frac{(t_2^{1-(1-\gamma)q}-t_1^{1-(1-\gamma)q})}{1-(1-\gamma)q}.
\end{array}
$$
Setting $\rho= t_2/t_1>1$, we get: 
$$
\frac{M(1-(1-\gamma)q)}{ \gamma^q}\leq \frac{(\rho^{1-(1-\gamma)q}-1)(\rho-1)^{q-1}}{(\rho^\gamma-1)^q}.
$$
As the map $s\to \ln(\rho^s-1)$ is strictly concave on $(0,+\infty)$ (for $\rho>1$), the right-hand side of the above inequality is less than 1. This is in contradiction with the choice of $M$ in \eqref{choixMG}. So $\xi_0$ is optimal for $J$.

Next we compute the optimal trajectories for $u(0,x_0)$, where $x_0\in(0,1)$.
We note that such a minimizer must be of the form 
\begin{equation}\label{xitheta}
\xi_{\theta,x_0} (t)= \begin{cases}\ds
\frac{\theta^\gamma-x_0}{ \theta}t+x_0 &\mbox{in}\quad [0,\theta]
\vspace{.5mm}
\\
t^\gamma & \mbox{in}\quad (\theta, 1]
\end{cases}\qqf t\in[0,1],
\end{equation}
where $\theta$ minimizes 
$$
J[\xi_{\theta,x_0} ]= M \frac{|\theta^\gamma-x_0|^p}{q\theta^{p-1}}+ \int_\theta^1 \frac{1}{q}(\gamma t^{\gamma-1})^q dt.
$$
Computing the values where the derivative of this expression vanishes shows that the optimum is uniquely reached at a point of the form 
$\bar \theta= \sigma x_0^{1/\gamma}$, where $\sigma>1$ depends on $q$, $M$ and $\gamma$. \\

We now study minimizers for $u(t,x)$ where $(t,x)\in (0,1)\times (0,1)$ with $x>t^\gamma$. Since the  family of  curves $\xi_{\bar \theta, x_0}$, for $x_0\in (0,1)$ and $\bar \theta= \sigma x_0^{1/\gamma}$, covers the set $O= \{(t,x)\;|\; 0<t<x^{1/\gamma}<1\}$, the optimal trajectory for $u(t,x)$ is just the restriction to $[t,1]$ of the trajectory $\xi_{\bar \theta, x_0}$ for the unique $x_0\in (0,1)$ such that $\xi_{\bar \theta, x_0}(t)=x$, i.e., such that 
\be\label{defx0}
x=\frac{\bar \theta^\gamma-x_0}{ \bar\theta}t+x_0,\qquad \bar \theta =\sigma x_0^{1/\gamma}.
\ee
In particular, if we set 
\begin{equation*}
\xi_{ \theta,t,x} (s)= \begin{cases}\ds
\frac{ \theta^\gamma-x}{ \theta-t }(s-t)+x &\mbox{in}\quad [t,  \theta]
\vspace{.5mm}
\\
s^\gamma & \mbox{in}\quad ( \theta, 1]
\end{cases}\qqf s\in[t,1]
\end{equation*}
then 
$$
u(t,x)= \inf_{\theta\in (0,1)} J[\xi_{ \theta,t,x}, t,x] = \inf_{\theta\in (0,1)}  M\frac{|\theta^\gamma-x|^q}{q(\theta-t)^{q-1}}+ \int_\theta^1\frac{1}{q}(\gamma t^{\gamma-1})^q dt ,
$$
and the unique optimum is reached by $\bar \theta=\sigma x_0^{1/\gamma}$ with $x_0$ such that \eqref{defx0} holds. 
By standard arguments, we have
$$
\partial_t u(t,x)=\frac{\partial}{\partial t}J[\xi_{\bar \theta}, t, x]\, ,\quad D u(t,x)=\frac{\partial}{\partial x}J[\xi_{\bar \theta}, t, x]\, .
$$
An easy computation shows that
$$
\partial_t u(t,x)=- \frac{M(q-1)}{q} \frac{(\bar \theta^\gamma-x)^q}{(\bar \theta-t)^q}= -Cx_0^{-q(1/\gamma-1)}
$$
where $C>0$ depends on $q$,  $M$ and $\gamma$.

As, from \eqref{defx0}, the map $t\to x_0$ is decreasing on $(0, x^{1/\gamma})$  for any $x\in (0,1)$, the map $t\to |\partial_t u(t,x)|$ is increasing on $(0, x^{1/\gamma})$. Recalling the notation $O= \{(t,x)\;|\; 0<t<x^{1/\gamma}<1\}$, we have  
$$
\begin{array}{rl}
\ds \int_{(0,1)\times (0,1)} |\partial_t u(t,x)|^{1+\vep}\; \geq & \ds  \int_{O} |\partial_t u(t,x)|^{1+\vep}\; \geq\;  \int_{O} |\partial_t u(t,0)|^{1+\vep} \\
 = & \ds \int_0^1 \int_0^{x_0^{1/\gamma}} Cx_0^{-q(1/\gamma-1)(1+\vep)}dtdx_0\;= \; \ds  \int_0^1 Cx_0^{-q(1/\gamma-1)(1+\vep)+1/\gamma}dx_0,\\
\end{array}
$$
where the right-hand side diverges if $\vep\geq \frac{\gamma(q+1)+1-q}{q(1-\gamma)}$. So, if $\vep\geq \frac{\gamma(q+1)+1-q}{q(1-\gamma)}$, then  $\partial_t u\notin L^{1+\vep}((0,1)\times (0,1))$ and $Du\notin L^{p(1+\vep)}((0,1)\times (0,1))$ since 
$$
|Du(t,x)|= \left(pM\right)^{1/p}\left|\partial_tu(t,x)\right|^{1/p}\qquad {\rm in }\; O.
$$

To conclude, we note that $u_n$ defined in \eqref{defun} is given by the representation formula: 
\begin{equation*}\label{eq:un}
u_n(x,t)=\inf\left\{\int_t^1\frac{a_n(s,\xi(s))}{q}|\xi'(s)|^qds+g_n(\xi(1))\ : \xi\in W^{1,q}([t,1])\,,\;\xi(t)=x\right\}\,.
\end{equation*}
In particular $(u_n)$ converges uniformly to $u$. Fix  $\vep\geq \frac{\gamma(q+1)+1-q}{q(1-\gamma)}$. Then, as $\partial_tu \notin L^{1+\vep}((0,1)\times (0,1))$ (resp. $Du\notin  L^{p(1+\vep)}((0,1)\times (0,1))$), the sequence $(\partial_t u_n)$ (resp. $(Du_n)$) cannot be bounded in $L^{1+\vep}((0,1)\times (0,1))$ (resp.  in  $L^{p(1+\vep)}((0,1)\times (0,1))$). 
\end{proof}

\begin{Remark}\label{remrem}{ \rm
It is important to note that the Sobolev estimates obtained in this article are not true in general for a.e. solutions of (\ref{HJ0}). Indeed,  let us consider  $H$ and $f$  two continuous functions  such that $\xi\to H(t,x,\xi)$ is a convex function which is coercive in a direction $\lambda\in \R^d$, i.e.  for every  $x\in Q_1$ and every bounded set $K$ of $\R\times\R^d$, there exist constants $m,c>0$, such that
$$
H(t,x,\xi+s\lambda)\geq m |s|- c
$$ 
for every $s\in\R,x\in Q_1$ and for every $(t,\xi)\in K$. Then, the set of solutions of
$$
-\partial_t u +H(t,x,Du)=f(t,x)\quad \text{ a.e.  in }  (0,1)\times Q_1
$$ 
is dense, in the $L^\infty$ norm, in the set of subsolutions of
\be\label{esubsol}
-\partial_t u +H(t,x,Du)\leq f(t,x)\quad \text{ a.e.  in }  (0,1)\times Q_1\; .
\ee
This result follows from  Theorem 2.3 in \cite{DaMa}.

Suppose now that the Hamiltonian $H$ satisfies in addition growth conditions (\ref{igrowth}).
Thanks to the above result, if Theorem  \ref{th:main} were true for a.e. solutions then the same should hold for a.e. subsolutions of (\ref{esubsol}). In particular, an a.e. subsolution $u$ of (\ref{HJ0}) should be in $W^{1,1}_{loc}((0,1)\times Q_1)$. However, there exist a.e. subsolutions that are just functions of bounded variation and whose derivatives are singular measures.

For example, consider the function
$$
u(t,x):=\left\{ \begin{array}{cc}
0 & (t,x)\in (0,\frac{1}{2})\times Q_1 \\
1 & (t,x)\in (\frac{1}{2},1)\times Q_1.
\end{array} \right.
$$
Then, for all  $H:(0,1)\times\R^d\times\R^d\to\R$  continuous, that  satisfy (\ref{igrowth}) and  such that $\xi\to H(t,x,\xi)$ is convex and coercive in a direction $\lambda$,  we have that $u$ satisfies
$$
-\partial_t u +H(t,x,Du)\leq \bar C
$$ 
in the sense of distributions. Nevertheless, $u$ cannot be in $W^{1,1}_{loc}((0,1)\times Q_1)$, since it is of bounded variation and $\partial_t u$ is a singular measure. 
 
This reflects the fact that  Sobolev regularity is specific to viscosity solutions of Hamilton-Jacobi equations.

}
\end{Remark}

\section{Application to Mean Field Games}\label{sec:MFG}

In this section we apply our main result to first order Mean Field Games systems (MFG). The system studied in this section takes the form:
\be\label{MFG}
\left\{\begin{array}{cl}
(i)&- \partial_t u +H(x,Du) =g(x,m(t,x))\qquad {\rm in }\; (0,T)\times \T^d\\
(ii) & \partial_t m-{\rm div} (mD_pH(x,Du))=0\qquad {\rm in }\; (0,T)\times \T^d\\
(iii)& m(0,x)=m_0(x), \; u(T,x)=u_T(x)\qquad {\rm in }\;  \T^d
\end{array}\right.
\ee
where $\T^d= \R^d/\Z^d$ is $d-$dimensional torus, the Hamiltonian $H:\T^d\times \R^d\to\R$ is convex in the second variable, the coupling
$g:\T^d\times [0,+\infty)\to [0,+\infty)$ is increasing with respect to the second variable, $m_0$ is a smooth probability density and $u_T:\T^d\to\R$ is a smooth given function.  

Let us briefly recall that the MFG systems were introduced by Lasry and Lions in \cite{lasry06, lasry06a, lasry07} (see also the works by Huang, Caines and Malhamé \cite{huang2007large} for closely related approach) to describe differential games with infinitely many agents. In this system equation \eqref{MFG}-(i) is the equation for the value function $u$ of a typical small agent whose cost depends on the point-wise density of all the agents through the (local) coupling $g$. The evolution equation \eqref{MFG}-(ii) describes the evolution of the agents when they control their own dynamics in an optimal way. In \cite{lions07} the authors explain that system \eqref{MFG} has smooth solutions under strong structure and regularity conditions on the data. 

In \cite{cg} (see also \cite{c1}) it is proved that, under less demanding assumptions on the data, the MFG system has a unique weak solution. Our goal is to show that the equation is actually satisfied in a stronger sense than stated in \cite{cg}. Let us first state our hypothesis.  
\begin{itemize}
\item the coupling $g:\T^d\times [0,+\infty)\to \R$ is continuous in both variables, strictly increasing with respect to the second variable $m$, 
and there exist $r>1$  and $C_1$  such that 
\be\label{Hypf}
 \frac{1}{C_1}|m|^{r'-1}-C_1\leq g(x,m) \leq C_1 |m|^{r'-1}+C_1 \qquad \forall m\geq 0 \;.
\ee
where $r'$ is the conjugate exponent of $r$. 
Moreover we ask the following normalization condition to hold:
\be\label{Hypf(0,m)=0}
g(x,0)=0 \qquad \forall x\in \T^d\;.
\ee

\item The Hamiltonian  $H:\T^d\times \R^d\to\R$ is continuous in both variables, convex and differentiable in the second variable, with $D_pH$ continuous in both variable, and has a superlinear growth in the gradient variable: there exist $p>0$ and  $C_2>0$ such that $r> 1+ d/p$
and
\be\label{HypGrowthH}
\frac{1}{pC_2} |\xi|^{p} -C_2\leq H(x,\xi) \leq \frac{C_2}{p}|\xi|^p+C_2\qquad \forall (x,\xi)\in \T^d\times \R^d\;.
\ee

\item $u_T:\T^d\to \R$ is of class ${\mathcal C}^1$, while $m_0:\T^d\to \R$ is a continuous density, with $m_0\geq 0$ and $\ds \int_{\T^d} m_0dx=1$. 
\end{itemize} 

As usual we denote by $q$ the conjugate exponent of $p$. Here is the main existence/uniqueness result of \cite{cg}. 

\begin{Theorem}[\cite{cg}]\label{def:weaksolMFG} There is a unique weak solution of \eqref{MFG}, i.e., a unique pair $(m, u)\in L^r((0,T)\times\T^d) \times BV((0,T)\times\T^d)$ such that
\begin{itemize}
\item[(i)] $u$ is continuous in $[0,T]\times \T^d$, with  
$$\ds Du\in L^p, \; \ds mD_pH(x,Du)\in L^1
\; {\rm and}\; 
\; \left(\partial_tu^{ac}- \lg Du, D_pH(x,Du)\rg\right)m \in L^1\;.$$ 

\item[(ii)] Equation \eqref{MFG}-(i) holds in the following sense:
\be\label{eq:ae}
\ds \quad -\partial_t u^{ac}(t,x) +H(x,Du(t,x))= g(x,m(t,x)) \quad \; \mbox{\rm a.e. in $\{m>0\}$}
\ee
(where $\partial_t u^{ac}$ is the absolutely continuous part of the measure $\partial_t u$ with respect to the Lebesgue measure)
 and  inequality
\be\label{eq:distrib}
\quad -\partial_t u +H(x,Du)\leq  g(x,m) \quad {\rm in}\; (0,T)\times \T^d 
\ee holds in the sense of distributions, with $u(T,\cdot)=u_T$ in the sense of trace,  

\item[(iii)] Equation \eqref{MFG}-(ii) holds:  
\be\label{eqcontdef}
\ds \quad \partial_t m-{\rm div}( mD_pH(x,Du))= 0\quad {\rm in }\; (0,T)\times \T^d, \qquad m(0)=m_0
\ee
in the sense of distributions,

\item[(iv)] The following equality holds: 
\be\label{defcondsup} 
\int_0^T\int_{\T^d} m\left(\partial_tu^{ac}- \lg Du, D_pH(x,Du)\rg \right)= \int_{\T^d} m(T)u_T-m_0u(0). 
\ee 
\end{itemize}
By uniqueness we mean that $m$ is indeed unique and $u$ is uniquely defined in $\{m>0\}$.
\end{Theorem}

In \cite{cg} the above result is stated under more general conditions.  Under the above assumptions, it is explained in Remark 3.7 of  \cite{cg} (see also \cite{c1}) that $u$ is Hölder continuous. It is also globally unique (not only in $\{m>0\}$) if one requires the additional condition 
\be\label{subsoll}
-\partial_t u+H(x,Du)\geq 0\qquad {\rm in}\; (0,T)\times \T^d
\ee
in the viscosity sense. 

Theorem \ref{th:main}   implies that $u$  actually belongs to  a Sobolev space: 
 
\begin{Corollary} Let $(u,m)$ be the unique weak solution of \eqref{MFG} which satisfies \eqref{subsoll}. Then $u$ belongs to $W^{1,1}_{loc}((0,T)\times \T^d)$, $u$ is differentiable almost everywhere and the following equality holds:
$$
\ds -\partial_t u(t,x) +H(x,Du(t,x))= g(x,m(t,x)) \quad \; \mbox{\rm a.e. in $(0,T)\times \T^d$}.
$$
\end{Corollary}

In particular equation \eqref{MFG}-(i) is satisfied in a strong sense. 

\begin{proof} Let us set 
$$
G(x,m)= \left\{\begin{array}{ll}
\ds  \int_0^m g(x,\rho)d\rho & {\rm if } \; m\geq 0\\
 +\infty & {\rm otherwise}
 \end{array}\right.
 $$
 and let $G^*$  the convex conjugate of $G$ with respect to the second variable, i.e., 
$$
G^*(x,a) = \sup_{m\geq 0} am-G(m).
$$
In \cite{c1, cg} the construction of the $u-$component of the solution is obtained as follows. Let $(u_n)$ be a minimizing sequence of the following problem: 
$$
\inf \int_0^T\int_{\T^d} G^*(x,-\partial_tw(t,x)+H(x,Dw(t,x)))\ dxdt - \int_{\T^d} m_0(x) w(0,x)dx
$$
where the infimum is taken over all ${\mathcal C}^1$ maps $w$ such that $w(T,\cdot)=u_T$. Let us set 
$$
\alpha_n(t,x):=-\partial_t u_n(t,x)+H(x,Du_n(t,x)). 
$$
Then it is proved in \cite{c1, cg} that one can build $\alpha_n$ in such a way that $\alpha_n\geq 0$, $(\alpha_n)$ is bounded in $L^r$ and $(u_n)$ is uniformly bounded. Moreover $(u_n)$ uniformly converges to the $u-$component of the solution of \eqref{MFG}. 

In view of the above estimates on $(\alpha_n)$ and $(u_n)$,  our main result Theorem \ref{th:main} implies that, locally in $(0,T)\times \T^d$, there exists $\vep>0$ such that $(\partial_t u_n)$ is bounded in $L^{1+\vep}$ while $Du_n$ is bounded in $L^{p(1+\vep)}$. By weak convergence, $u$ is still in $W^{1,1}$. 

Then, by \eqref{eq:ae}, the following equality holds 
$$
\ds -\partial_t u(t,x) +H(x,Du(t,x))= g(x,m(t,x)) \quad \; \mbox{\rm a.e. in $\{m>0\}$}.
$$
As $u$ satisfies  \eqref{eq:distrib} in the sense of distributions and \eqref{subsoll} in the viscosity sense,   Proposition \ref{AeDiff} implies that $u$ is a.e. differentiable. Hence inequalities
$$
\ds -\partial_t u(t,x) +H(x,Du(t,x))\leq  g(x,m(t,x)) \quad \; \mbox{\rm a.e. in $(0,T)\times \T^d$}
$$
and 
$$
-\partial_t u(t,x) +H(x,Du(t,x))\geq 0 \quad \; \mbox{\rm a.e. in $(0,T)\times \T^d$}
$$
hold. As $g(x,0)=0$, we finally get 
$$
\ds -\partial_t u(t,x) +H(x,Du(t,x))= g(x,m(t,x))
$$
a.e. in $\{m=0\}$.
\end{proof}


\end{document}